\documentclass{article}
\usepackage[utf8]{inputenc}      		% Required for simple use of ö,ä,ü
\usepackage{graphicx}					% Required for use of figures
\usepackage{textcomp}					% Required to write ° with \textdegree
\usepackage{amsmath}% Required for advanced formulas
\usepackage{amssymb}
\usepackage{siunitx}% Required for scientific notation of numbers (and more)
\usepackage{amsthm}
\usepackage{subcaption}					% Required for subfigures
\usepackage{lettrine}					% Required for large letters at beginning of chapter
\usepackage{sectsty}					% Required for centered section captions 
\usepackage{fancyhdr}					% Required for modified title page
\usepackage[export]{adjustbox}
\usepackage{xcolor}
\usepackage{geometry}
\usepackage{tcolorbox}

\usepackage{tabularx}
\usepackage{algorithm}
\usepackage{algpseudocode}
\usepackage[obeyspaces]{url}
\usepackage{hyperref}

\newcommand{\F}{\mathbb F}
\newcommand{\N}{\mathbb N}

\newcommand{\R}{\mathbb R}

\newcommand{\cA}{\mathcal{A}}

\newcommand{\cM}{\mathcal{M}}
\newcommand{\cN}{\mathcal{N}}
\newcommand{\cO}{\mathcal{O}}
\newcommand{\cP}{\mathcal{P}}
\newcommand{\cS}{\mathcal{S}}

\newcommand{\transpose}{{\rm{T}}}

\DeclareMathOperator*{\argmin}{arg\,min}
\DeclareMathOperator{\spn}{span}

\DeclareMathOperator{\sym}{\text{sym}}

\DeclareMathOperator{\St}{\text{St}}

\DeclareMathOperator*{\esssup}{ess\,sup}

\numberwithin{equation}{section} %Für ordentliche Numerierung der Gleichungen

\newcommand{\propnumber}{} % initialize
\newtheorem*{cor}{Corollary \propnumber}

\newtheorem{theorem}{Theorem}[section]
\newtheorem{definition}[theorem]{Definition}
\newtheorem{lemma}[theorem]{Lemma}

\newtheorem{proposition}[theorem]{Proposition}
\newtheorem{corollary}[theorem]{Corollary}

\theoremstyle{definition}

\newcommand{\vertiii}[1]{{\left\vert\kern-0.25ex\left\vert\kern-0.25ex\left\vert #1 
    \right\vert\kern-0.25ex\right\vert\kern-0.25ex\right\vert}}

\title{Kernel-based learning of manifold-to-manifold maps from
  scattered data}
\author{Daniel Fischer\\ Department of Mathematics\\ University of
  Bayreuth\\ 95440 Bayreuth \\ Germany\\daniel.fischer@uni-bayreuth.de
\and Holger Wendland\thanks{The work of HW has been supported by the DFG
  under  project no. 514588180.} \\ Department of Mathematics\\ University of
  Bayreuth\\ 95440 Bayreuth \\ Germany\\holger.wendland@uni-bayreuth.de}

\begin{document}

\maketitle

\begin{abstract}
   We describe and analyze new methods for approximating
   manifold-to-manifold maps using only scattered data information. 
   To this end, we first study kernel-based approximation
    methods for scalar-valued functions defined on a manifold and
    derive a new sampling inequality and error estimates for functions
    from Sobolev spaces. After that, these methods are combined with 
     a closest point projection to reconstruct manifold-to-manifold
     maps. The new methods are analyzed and error estimates are derived.
     Finally, numerical examples are given to verify the theoretical
     findings.


\end{abstract}

%%%%%%%%%%%%%%%%%%%%%%%%%%%%%%%%%%%%%%%%%%%%%%%
\section{Introduction}
Interpolation and approximation  with positive definite kernels provides a powerful
framework for reconstructing a multivariate function 
$f : \Omega \to \R$
from known function values at a finite set of data sites scattered
throughout the domain of definition $\Omega\subseteq \R^d$. Since these
methods were first proposed, they have been generalized in various
directions.  
On the one hand, numerous techniques have been developed for defining
positive definite kernels on more specific domains of definition, such
as smooth manifolds $\cM$, thereby enabling the reconstruction
of functions $f : \cM \to \R$ (see, for example,
\cite{FW20, FW29, FW24, FW14}).  
On the other hand, the theory of positive definite kernels and the
associated reproducing kernel Hilbert spaces has been extended from
the scalar-valued setting to vector- and, more generally, Hilbert
space-valued functions, thus allowing the approximation of mappings $f
: \Omega \to W$, where $W$ is a Hilbert space \cite{ VVK1, VVK2, VVK3,
  GieslWendland}.  

Motivated by applications in numerous scientific and engineering
domains, recent years have seen a growing interest in approximation
methods for functions $f : \Omega \to \cN$, whose codomain
$\cN$ possesses the structure of a smooth manifold rather than
that of a vector space.  
Such mappings arise, for instance, in parametric model order reduction
\cite{ZimmermannOverview}, where the codomain typically is the Stiefel
manifold $\St(n,r)$ of column orthogonal matrices in $\R^{n\times r}$
or the Grassmann manifold $\text{Gr}(k,n)$ of all $k$-dimensional
subspaces of an $n$ dimensional linear space, in interpolation of
diffusion tensor images created from magnetic resonance imaging
\cite{PennecUebersichtSPD}, where the codomain is the manifold of
symmetric positive definite matrices, in the context of Cosserat-type
material models \cite{Neff_Cosserat}, where
$\cN=\R^3\times\text{SO}(3)$, or in crystallographic texture
analysis \cite{HielscherLippert}, where the codomain can be seen as
quotient $\text{SO}(3)/\cS$ with some finite symmetry group
$\cS$.  

Many of the proposed manifold-valued approximation techniques have
already been analyzed mathematically and can be broadly categorized
into three classes: 
\begin{enumerate}
    \item Projection-based methods \cite{GawlikLeokI, ProjectionFE,
      HielscherLippert}, where first an interpolant is constructed in
      the ambient space of an embedded submanifold, and the resulting
      values are subsequently projected onto the manifold by means of
      a suitable mapping; 
    \item Tangent space methods (Push-Interpolate-Pull methods)
      \cite{Wendland-??-2,ZimmermannDataProc, ZimmermannMultivariate}, where the
      given data are first mapped to a linear space - typically the
      tangent space at a chosen base point - followed by classical
      vector-valued interpolation and finally mapping back the
      reconstructed values to the manifold via an inverse
      transformation; 
    \item Riemannian mean-based methods
      \cite{OptimalAPriori_GeodesicFE, Grohs_SDA}, where the
      evaluation of the interpolant involves computing a weighted mean
      of function values in $\cN$ through nonlinear
      optimization. 
\end{enumerate}
There also exist several techniques that do not fall directly into any
of the aforementioned categories (e.g., \cite{Hardering_Splines}), or
that combine multiple of the above approaches
\cite{MultipleTangentSpaces}. 

With the exception of \cite{Grohs_SDA}, however, these methods require
the sample locations to have a certain structure. In most cases, it is
assumed that the function values are known on some type of regular
grid. Without such a grid structure, these methods cannot be
applied. Fortunately, in this situation the ideas from kernel-based
interpolation of scalar-valued functions on manifolds, vector-valued
kernel-based interpolation and manifold-valued approaches can be
naturally combined into a method for reconstructing functions between
manifolds. In this work, we derive such  methods, i.e., we address the
problem of reconstructing a manifold-valued function 
$f : \cM \to \cN,$ 
where $\cM$ and $\cN$ are compact smooth manifolds,
from known values $f(x_1), \dots, f(x_N) \in \cN$ at a set of
scattered data sites $X = \{x_1, \dots, x_N\} \subseteq \cM$. To
our knowledge, this problem has not been investigated in this form
until now. In principle, various kernel-based approximation techniques
could be combined with any of the aforementioned methods for
manifold-valued interpolation. However, we restrict our attention to
kernel-based interpolation (see \cite{Fuselier-Wright-12-1}) and
penalized least-squares approximation on embedded submanifolds $\cM$.
In both cases,  the kernel on $\cM$ is obtained by
restricting a positive definite kernel defined on the ambient space to
the submanifold. Moreover, we employ a projection approach to ensure
that the resulting interpolant indeed takes values in the target
manifold $\cN$. 
%For a comprehensive study of the combination
%of kernel-based approximation schemes with other techniques for
%manifold-valued interpolation (e.g., tangent space approaches and
%approximation via Riemannian means) we refer the interested reader to
%the forthcoming dissertation \cite{Dissertation}. 

The remainder of this article is organized as follows: In the next
section, we discuss kernel-based approximation methods for
scalar-valued functions, defined on some domain $\Omega\subseteq\R^d$.
 In the third section, we extend these results to
kernel-based approximations for scalar-valued functions defined on embedded
submanifolds. To this end, we first collect some necessary material on
embedded submanifolds in Subsection \ref{Submanifolds} and then state and prove
results on the norm-minimal interpolation and penalized least-squares
processes. While we can rely on previous work for the interpolation
operator, the results on the penalized least-squares approximation are
new in this context. Particularly, we discuss their approximation
properties if the measured data contain noise.
Section \ref{Together_Section} then contains the core of this paper,
where the kernel-based approximation methods are combined with a
closest point projection to define  manifold-to-manifold
approximations. Finally, in Section \ref{Numerics},
we numerically illustrate these results with two academic examples and a
real-world application from crystallography.

%%%%%%%%%%%%%%%%%%%%%%%%%%%%%%%%%%%%%%%%%%%%%%%%
\section{Kernel-based approximation}\label{Klassisch_Skalar}

Kernel-based learning or approximation is a well-established tool in
approximation theory and statistics, it comes in the form of scattered
data approximation with radial basis functions (see for example
\cite{Buhmann-00-1, Fasshauer-07-1, Wendland-05-1}), Gaussian processes (see for example
\cite{Rasmussen-Williams-06-1}) and support vector machines (see for
example \cite{Schoelkopf-Smola-02-1, Steinwart-Christmann-08-1}).

Usually, in this context, functions $f:\Omega\to\R$ that are defined
on some domain $\Omega\subseteq\R^d$ are learned or approximated using
only information $f(x_1),\ldots,f(x_N)$ measured at scattered data
locations $X=\{x_1,\ldots,x_N\}\subseteq\Omega$. In this section, we
will shortly review the material 
relevant to us. Here, we follow mainly \cite{Wendland-05-1}.
Kernel-based learning of functions $f:\Omega\to\R$ usually assumes
that the unknown function comes from a reproducing kernel Hilbert
space $H$. This is a Hilbert space of real-valued functions defined on
$\Omega$, possessing a kernel $K:\Omega\times\Omega\to\R$ satisfying
$K(\cdot,x)\in H$ for all $x\in\Omega$ and $f(x)=\langle f, K(\cdot,x)\rangle_H$
for all $x\in \Omega$ and  all $f\in H$. Obviously, this means that
point evaluations $\delta_x:H\to \R$, $x\mapsto f(x)$ are continuous
mappings and their Riesz representer is given by the kernel
$K(\cdot,x)$. The reproducing kernel of a Hilbert space is uniquely
determined by the Hilbert space and is positive (semi-)definite in the
sense that all the matrices $K(X,X) = (K(x_i,x_j))_{ij}\in\R^{N\times
  N}$ are positive (semi-)definite for any set $X=\{x_1,\ldots,x_N\}\subseteq\Omega$ of $N\in\N$
distinct points. 

\begin{definition}\label{approxOperators} 
 Let $H$ be a reproducing kernel Hilbert space of functions
 $f:\Omega\to\R$. Let $X=\{x_1,\ldots,x_N\}\subseteq\Omega$ be a set
 of pairwise distinct points.
 \begin{enumerate}
 \item The optimal recovery or  norm-minimal interpolant $s_0$
   is the solution of
   \[
   \min \left\{ \|s\|_H : s\in H \mbox{ with } s(x_i)=f(x_i), 1\le i\le N\right\}.
   \]
 \item Given a parameter $\lambda>0$, the penalized least-squares approximation
   $s_\lambda$ is the solution of
   \[
  \min\left\{ \sum_{i=1}^N |f(x_i)-s(x_i)|^2 + \lambda\|s\|_H^2 : s\in
  H \right\}.
  \]
 \end{enumerate}
\end{definition}
In the context of support vector machines, the least-squares loss
function $L(w,y)=|w-y|^2$, which is used in the definition of the
penalized least-squares problem is often replaced by other loss
functions. This is, in principle, also possible in this context, but
for simplicity, we will only deal with the  least-squares loss.

It is well-known that the second minimization problem always has a
unique solution, while the first one has a unique solution if the
reproducing kernel $K:\Omega\times\Omega\to\R$ is positive
definite.

\begin{lemma}\label{lem:kernelsolution}
Let $H$ be a reproducing kernel Hilbert space of functions
$f:\Omega\to\R$ with positive definite kernel
$K:\Omega\times\Omega\to\R$. Then the problems from
Definition \ref{approxOperators} have unique solutions $s_\lambda$,
$\lambda\ge 0$. They can be written as
\[
s_\lambda = \sum_{j=1}^N \alpha_j K(\cdot,x_j), \qquad \lambda\ge 0,
\]
where the coefficients $\alpha=(\alpha_1,\ldots,\alpha_N)^\transpose$
are the solution of the linear system
$
(K(X,X) + \lambda I) \alpha = f(X)
$
with  $f(X) = (f(x_1),\ldots, f(x_N))^\transpose$ . This defines
linear mappings $I_X, Q_{X,\lambda}: C(\Omega) \to
V_X:=\spn\{K(\cdot,x_1),\ldots, K(\cdot,x_N)\}$ by setting $I_Xf:=s_0$
and $Q_{X,\lambda}f:=s_\lambda$.
\end{lemma}

In this paper, we will only consider reproducing kernel Hilbert
spaces, which are Sobolev spaces and all error analysis will take
place in Sobolev spaces, which we introduce now, mainly to fix our
notation. Let $\Omega\subseteq \R^d$ be open. For $l\in \N_0$ and $1\leq
p\leq \infty$, we define the Sobolev space $W_p^l(\Omega)$ of order
$l$ to consist of all functions $f\in L_p(\Omega)$ with weak derivatives $D^\alpha
f\in L_p(\Omega)$ for all $\alpha\in\N_0^d$ with $|\alpha|\le l$. As
usual, this space is equipped with the norm
\[
\|f\|_{W_p^l(\Omega)}:=\begin{cases}\left(\sum_{|\alpha|\leq l}\|D^\alpha f\|_{L_p(\Omega)}^p\, \right)^{\frac{1}{p}},&1\leq p <\infty,\\
\max_{|\alpha|\leq l}\|D^{\alpha} f\|_{L_\infty(\Omega)}, & p=\infty.
\end{cases}
\]
We will also consider fractional order Sobolev spaces. For $\tau=l+t$,
$0<t<1$, $1\leq p< \infty$, these are defined as all functions $f\in
W_p^l(\Omega)$ satisfying   $\|f\|_{W_p^\tau(\Omega)}<\infty$,
where 
\[
\|f\|_{W_p^\tau(\Omega)}:=\left(\|f\|^p_{W_p^l(\Omega)}+
\sum_{|\alpha|=l}\int_{\Omega}\int_{\Omega}\frac{|D^{\alpha}f(x)-D^{\alpha}f(y)|^p}{\|x-y\|_2^{d+tp}}\,dx\,dy\right)^{\frac{1}{p}}.     
\]
Here, , $\|\cdot\|_2$ denotes the standard Euclidean norm on $\R^d$.
As usual, in the case of $p=2$ we also write $H^\tau(\Omega):=W_2^\tau(\Omega)$.

The Sobolev embedding theorem guarantees that for any $\tau>d/2$ and any
$\Omega\subseteq\R^d$ with at least a Lipschitz boundary, the
Sobolev space $H^\tau(\Omega)$ is a reproducing kernel Hilbert
space. However, only for very specific $\Omega$ the reproducing kernel
is explicitly known. In the case $\Omega=\R^d$, we can use Fourier
transform to introduce an alternative inner product with equivalent norm on
$H^\tau(\R^d)$ with corresponding reproducing kernel $K_\tau:\R^d\times\R^d\to\R$.
According to \cite[Corollary
    10.13]{Wendland-05-1}, it is possible to choose the kernel as a
translation-invariant function, i.e.
 $K_\tau(x,y)=\Phi_\tau(x-y)$ with integrable  $\Phi_\tau:\R^d\to\R$ having a 
  Fourier transform
  \[
  \widehat{\Phi_\tau}(\omega) = \int_{\R^d} \Phi_\tau (x) e^{-ix^\transpose\omega} dx
  \]
that decays like  $(1+\|\cdot\|_2^2)^{-\tau}$, i.e. there are two constants
  $c_1,c_2>0$ such that
  \begin{equation}\label{ftdecay}
c_1(1+\|\omega\|^2_2)^{-\tau} \le \widehat{\Phi_\tau}(\omega) \le
c_2(1+\|\omega\|_2^2)^{-\tau}, \qquad \omega\in \R^d.
  \end{equation}
  It is also often possible to choose
  the function $\Phi_\tau$ to be radial, i.e. $\Phi_\tau=\phi_\tau(\|\cdot\|_2)$ with
  a function $\phi_\tau:[0,\infty)\to\R$.
For example, the Matern functions are of this form with
$\widehat{\Phi_\tau}(\omega)=(1+\|\omega\|_2^2)^{-\tau}$. Other examples
are Wendland's  compactly supported radial basis functions
$\Phi_{d,k}=\phi_{d,k}(\|\cdot\|_2)\in C^{2k}(\R^d)$ 
with  $\tau=k+\frac{d+1}{2}$,  see \cite{Wendland-98-1}.

As outlined in \cite{Giesl-etal-21-1, Wendland-05-1} this can be used
to introduce an alternative inner product with equivalent norm also on
$H^\tau(\Omega)$  such that the 
restricted kernel $\Psi_\tau(x,y):=\Phi_\tau(x-y)$, $x,y\in\Omega$, is the
reproducing kernel of $H^\tau(\Omega)$ with respect to this inner
product. To be more precise, Lemma 2.2 in \cite{Giesl-etal-21-1}
states the following.

\begin{proposition}\label{locns}
  Assume $\Phi_\tau\in L_1(\R^d)\cap C(\R^d)$ has a Fourier
  transform $\widehat{\Phi_\tau}$ satisfying (\ref{ftdecay}) with
  $\tau>d/2$. Let $\Omega\subseteq\R^d$ be a bounded domain with a
  Lipschitz boundary. Let $\Psi_\tau:\Omega\times\Omega\to\R$ be
  defined by $\Psi_\tau(x,y)=\Phi_\tau(x-y)$, $x,y\in\Omega$. Then, there
  exists an inner product $\langle
  \cdot,\cdot\rangle_{\Psi_\tau}:H^\tau(\Omega)\times
  H^\tau(\Omega)\to\R$ on  $H^\tau(\Omega)$ such that
  $\Psi_\tau$ is the reproducing
  kernel of $H^\tau(\Omega)$ with respect to this inner product. The
  norm $\|\cdot\|_{\Psi_\tau}$ induced by this inner product is
  equivalent to the  norm on $H^\tau(\Omega)$, i.e. there
  are constants $C_1,C_2>0$ such that
  \[
  C_1 \|u\|_{\Psi_\tau} \le \|u\|_{H^\tau(\Omega)} \le C_2
  \|u\|_{\Psi_\tau}, \qquad u\in H^\tau(\Omega).
  \]
\end{proposition}
In what follows, we will, if not stated otherwise, always consider
$H^\tau(\Omega)$ to be equipped with the norm $\|\cdot\|_{\Psi_\tau}$
and will not distinguish between this norm and  $\|\cdot\|_{H^\tau(\Omega)}$.

We  end this section by stating  required sampling inequalities,
which were first introduced in the context of scattered data
approximation in \cite{NarcowichWardWendland} and later slightly generalized in
\cite{Arcangeli}.  To formulate them, we need
the so-called fill distance or mesh norm
\[
   h_{X,\Omega}:=\sup_{x\in \Omega} \min_{x_i\in X}
   \|x-x_i\|_2,
\]
 which measures how well the data sites $X$ cover the domain
 $\Omega$. Other relevant quantities are the separation radius
 $q_X:=1/2\min_{i\ne j} \|x_i-x_j\|_2$ and their ratio
 $\rho_{X,\Omega}:=h_{X,\Omega}/q_X$, which measures how uniformly
 distributed the points are. In the following theorem, we use the
 notation $(t)_+=\max\{t,0\}$ for any $t\in\R$ and
 \[
 \|v\|_{\ell_p(X)} = \begin{cases}
   \left(\sum_{j=1}^N |v(x_j)|^p \right)^{1/p} & \mbox{ for } 1\le
   p<\infty,\\
   \max_{1\le j\le N} |v(x_j)| & \mbox { for } p=\infty.
 \end{cases}
   \]
  for any $v:\Omega\to\R$ and
  $X=\{x_1,\ldots,x_N\}\subseteq\Omega\subseteq\R^d$. Later on, we
  will extend this notation to $v:\Omega\to\R^{\ell}$ in the obvious
  way, by interpreting $(v(x_1),\ldots v(x_N))$ as a vector in $\R^{N \ell}$.

\begin{theorem}[Sampling Inequality]\label{Sampling_Ineq}
  Let $\Omega\subseteq\R^d$ be a bounded domain with a Lipschitz 
boundary. Let $p,q\in [1,\infty]$ and $n \in \N$. Let $\tau>d/2$ and
$\gamma:=\max\{2,p,q\}$. Then, there exist constants $h_0>0$
(depending on $\Omega$ and $\tau$) and $C>0$ (depending on $\Omega$,
$\tau$ and $q$) such that for all $X=\{x_1,\ldots,x_N\}\subseteq 
\Omega$ 
with $h_{X,\Omega}\le h_0$ and all $v\in H^\tau(\Omega)$,
we have
\[
|v|_{W_q^\mu(\Omega)} \le C\left(h_{X,\Omega}^{\tau-\mu-d(1/2-1/q)_+} |v|_{H^\tau(\Omega)}
+ 
h_{X,\Omega}^{d/\gamma-\mu} \|v\|_{\ell_p(X)}\right) 
\]
for all $0\le \mu\le \ell(\tau,d,q)$, where
$\ell(\tau,d,q)=\ell_0(\tau,d,q):=\tau-d(1/2-1/q)_+$ in 
the case of $\tau\in\N$ and either $q>2$ and $\ell_0(\tau,d,q)\in\N$ or
$q=2$. Otherwise $\ell(\tau,d,q)=\lceil \ell_0(\tau,d,q)\rceil-1$. Moreover,
we require $\mu\in\N_0$ in the case of $q=\infty$.
\end{theorem}

%%%%%%%%%%%%%%%%%%%%%%%%%%%%%%%%%%%%%%%%%%%%%%%%%
\section{Approximation of functions defined on embedded
  submanifolds}
The goal of this section is to discuss and extend the results of the
previous section to the approximation of functions $f:\cM\to\R$, where
$\cM$ is an $m$-dimensional smooth, compact and connected embedded
submanifold of $\R^d$. 
To  this end, we will first remind the reader of some necessary material
on embedded submanifolds, including Sobolev spaces. Then, we will
discuss kernel approximation for 
functions defined on manifolds. After that, we will 
give error estimates for the interpolation and the penalized
least-squares operators and finally extend these results to the
approximation of functions $f:\cM\to\R^{\ell}$. 

\subsection{Embedded submanifolds}
\label{Submanifolds}
In this section, we will collect necessary material on
smooth embedded submanifolds and fix the notation for the rest of this
paper. Details on the topic can, for example, be found in  
\cite{LeeSmooth}. As usual, we describe a smooth $m$-dimensional
submanifold $\cM$ of $\R^d$ by a smooth atlas $\cA_{\cM}=\{(U_i,\varphi_i) :
i\in I\}$ of charts $(U_i,\varphi_i)$ with relatively open sets
$U_i\subseteq\cM\subseteq\R^d$ and homeomorphisms
 $\varphi_i:U_i\to B_i:=\varphi_i(U_i)\subseteq\R^m$ such that
$\varphi_i^{-1}$ is an immersion and   any two charts
$(U_i,\varphi_i)$ and $(U_j,\varphi_j)$ are compatible in the sense
that either $U_i\cap U_j = \emptyset$ or the transition map
    $\varphi_j\circ \varphi_i^{-1}:\varphi_i(U_i\cap U_j)\to \varphi_j(U_i\cap U_j)$ is a
$C^{\infty}$-diffeomorphism. Two smooth atlases
$\cA_{\cM}$ and $\cA_{\cM}^*$  are called compatible if
 each chart of $\cA_{\cM}$ is compatible with each chart of
$\cA_{\cM}^*$. 

In the rest of this paper, we will use the term submanifold in this
sense, i.e. we will mean an  embedded submanifold.

As we will only consider compact submanifolds, we can always find a
finite atlas, i.e. an atlas with a finite index set
$I=\{1,\ldots,L\}$. However, we also need some information on the
ranges $\varphi_j(U_j)$, when dealing with  Sobolev spaces on manifolds. In \cite[Lemma
  1]{Behzadan_Holst} it is shown that given a smooth 
atlas $\cA=\{(V_i,\varphi_i): i\in I\}$ of a compact
manifold $\cM$, there is a finite open cover $\{U_j\}_{1\leq j\leq L}$
of $\cM$ and associated coordinate charts $\{\psi_j\}_{1\leq j\leq L}$
such that the atlas $\cA_{\cM}=\{(U_j,\psi_j): 1\leq j\leq
L\}$ is compatible to $\cA$ and the image $B_j=\psi(U_j)$ of each
chart is an open ball in $\R^m$. 
\begin{definition}
    A smooth, finite  atlas $\cA_\cM = \{(U_j,\psi_j): 1\le j\le L\}$
    for an $m$-dimensional smooth, compact  submanifold
    $\cM\subseteq\R^d$ is called admissible if the image $B_j$ of each
    coordinate domain $U_j$ in the atlas under the 
    respective coordinate chart is an open ball in $\R^m$. 
\end{definition}

To finalize our discussion on charts,
we need one more thing. We often 
require a partition of unity associated to admissible atlases. The
following result comes from \cite[Theorem 30]{Behzadan_Holst}. 

\begin{lemma}
  Let $\cM$ be a compact manifold with admissible atlas
    $\cA_{\cM}=\{(U_j,\psi_j): 1\leq j\leq L\}$. Then, there
    exists an associated partition of unity, i.e., a set $\{\chi_j:
    1\leq j\leq L\}$ of smooth functions $\chi_j:\cM\to \R$ such that
    \begin{enumerate}
        \item[(i)] $0\leq \chi_j\leq 1$,
        \item[(ii)] $\sum_{j=1}^L\chi_j=1$,
        \item[(iii)] $\operatorname{supp}(\chi_j)=\overline{\{x\in \cM: \chi_j(x)\neq 0\}}\subseteq U_j$.
    \end{enumerate}
    We write $\cA_{\cM}=\{(U_j,\psi_j,\chi_j): 1\leq j\leq L\}$ and
    call $\cA_{\cM}$ an admissible augmented atlas. 
\end{lemma}

For any $p\in\cM$, the tangent space $T_p\cM$ at $p$ consists of all
vectors $\dot{\gamma}(0)$ if $\gamma\in
C^1((-\epsilon,\epsilon),\R^d)$ is a curve with
$\gamma((-\epsilon,\epsilon))\subseteq\cM$ and $\gamma(0)=p$.
%\[
%T_p\cM:=\{\dot\gamma(t_0) : \gamma\in
%C^1(I,\R^d),\  t_0\in I\subseteq \R \text{ interval, with
%}\gamma(I)\subseteq \cM,\gamma(t_0)=p\}\subseteq
%\R^d.
%\] 
$T_p\cM$ is an $m$-dimensional linear subspace of $\R^d$. Its
orthogonal complement $N_p\cM:=(T_p\cM)^\perp$ with respect to the
inner product of the ambient $\R^d$ is called the normal space at
$p$. Often, we do not only care about the tangent space at a single
point $p\in \cM$, but instead consider the totality of all tangent or
normal spaces to the manifold. The resulting objects 
$T\cM:=\{(p,v):p\in \cM, v\in T_p\cM\}$ and $N\cM:=\{(p,n):p\in \cM,
n\in N_p\cM\}$ are called tangent and normal bundle, respectively. It can be
shown that $T\cM$ and $N\cM$ are themselves embedded submanifolds of
$\R^d\times \R^d$ of dimension $2m$  and $d$, respectively. 
As an embedded submanifold of $\R^{d\times d}$, the normal bundle
carries a unique smooth structure itself. Under this structure the
so-called canonical projection $\pi_{N\cM}:N\cM\to \cM, (p,n)\mapsto
p$ 
is a smooth map.

Finally, to measure distances on an embedded submanifold $\cM\subseteq
\R^d$ one could obviously use the Euclidean distance of the ambient
$\R^d$, i.e., $\|p-q\|_2$ for $p,q\in \cM$. However, on a connected
manifold $\cM$ it is more natural to define the distance as
the length of the shortest curve on $\cM$ connecting $p$ and $q$. This
leads to the definition of the intrinsic distance $d_{\cM}:\cM\times
\cM\to \R$, defined as 
\[
d_{\cM}(p, q) := \inf \left\{ L(\gamma) :
\gamma : [a, b] \to \cM \text{ piecewise } C^1,\, \gamma(a) = p,\, \gamma(b) = q \right\}.
\]
Here, the length \( L(\gamma) \) of a piecewise \( C^1 \) curve \( \gamma \) is given by
$L(\gamma) := \int_a^b \left\| \dot{\gamma}(t) \right\|_{2} \, dt.$

We end this section by recalling Sobolev spaces on manifolds. There are
various equivalent ways to define Sobolev spaces
on embedded submanifolds (via charts, via the Laplace-Beltrami
operator, or via covariant derivatives); for  a comprehensive review and
comparison of some of these definitions see \cite{Behzadan_Holst}. As
in \cite{Fuselier-Wright-12-1}, we use a 
definition in terms of charts. For our later results, however, we
require a substantially more detailed analysis of these spaces. To this end, we rely primarily
on the comprehensive treatment in \cite{Behzadan_Holst}. We start by
defining Lebesgue spaces $L_p(\cM)$ on $\cM$.
\begin{definition}
Let $\cA_{\cM}=\{(U_j,\psi_j,\chi_j): 1\le j\le L\}$ be an
admissible augmented atlas of $\cM$ and $1\leq p< \infty$. The space
$L_p(\cM)$ is the completion of $C^{\infty}(\cM)$ with respect to the
norm 
\[
\|f\|_{L_p(\cM)}:=\sum_{j=1}^L\|(\chi_jf)\circ
\psi_j^{-1}\|_{L_p(\psi_j(U_j))}.
\]
For $p=\infty$ we define $L_\infty(\cM)$ to consist of all functions
$f:\cM\to\R$ with
\[
\|f\|_{L_\infty(\cM)}:=\esssup_{x\in\cM} |f(x)| <\infty.
\]
\end{definition}
It can be shown \cite[Theorem 86]{Behzadan_Holst} that this definition
is equivalent to the usual one based on the Riemannian volume measure
on $\cM$ (see \cite[Definition 29.2)]{Behzadan_Holst}) in the sense of
equivalent norms, and therefore is independent of the chosen
admissible augmented atlas $\cA_\cM$. Moreover, we obviously have
\[
\|f\|_{L_\infty(\cM)} \le \sum_{j=1}^L \|(\chi_j f)\circ
\psi_j^{-1}\|_{L_\infty(\psi_j(U_j))}.
\]

Similarly, Sobolev spaces on the
compact manifold $\cM$ are defined as follows \cite[Remark  54.1]{Behzadan_Holst}.

\begin{definition}\label{Sobolev_Spaces_M}
  Let $\cA_{\cM}=\{(U_j,\psi_j,\chi_j): 1\le j\le L\}$ be an
admissible augmented atlas of $\cM$, $1\leq p\le\infty$ and $\tau \geq
0$ (with $\tau\in\N_0$ for $p=\infty$). The Sobolev  space $W_p^\tau(\cM)$ is defined as 
\[
W_p^\tau(\cM):=\{f\in L_p(\cM):
\|f\|_{W_p^\tau(\cM)}:=\sum_{j=1}^L\|(\chi_jf)\circ
\psi_j^{-1}\|_{W_p^\tau(\psi_j(U_j))}<\infty\}.
\]
\end{definition}

At first glance, $W_p^\tau(\cM)$ seems to depend on the choice of the
admissible augmented atlas $\cA_{\cM}$. However, in Theorem 87 of
\cite{Behzadan_Holst} it is shown that for a second admissible
augmented atlas, which is compatible to $\cA_{\cM}$, algebraically the
same space arises and the respective norms are equivalent. Note that
in \cite{Behzadan_Holst} all results are initially formulated in terms
of distributions on $\cM$. However, as explained in Section 9.3 of that
article, this is equivalent to our Definition \ref{Sobolev_Spaces_M},
as all such distributions are regular and correspond to $L_p(\cM)$ functions.  

For later purposes, we have to consider these spaces in more detail. Naturally,
one may ask whether for every function $f \in W_p^\tau(\cM)$ the local
representation $f \circ \psi_j^{-1}$ belongs to the classical Sobolev
space $W_p^\tau(\psi_j(U_j))$ on the open ball
$\psi_j(U_j)\subseteq\R^m$. Contrary to claims made in some
references, this is not true in general. A simple counterexample can
be found in \cite[Remark 60]{Behzadan_Holst}. Fortunately, a slightly
weaker result \cite[Corollary 8]{Behzadan_Holst} holds true. To
formulate it, we define the local Sobolev space
$W_{loc}^{\tau,p}(\Omega)$ for $1< p<\infty$, $\tau\geq 0$ on a nonempty
open set $\Omega\subseteq\R^d$ by 
\[
W_{loc}^{\tau,p}(\Omega):=\{f\in D'(\Omega):  \varphi f\in W_p^\tau(\Omega)
\mbox{ for all } \varphi\in C_c^{\infty}(\Omega)\}.
\]
Here, $D'(\Omega)$ denotes the space of distributions on $\Omega$ as
introduced in \cite[Definition 50]{Behzadan_loc} and
$C_c^\infty(\Omega)$ consist of all $C^\infty(\Omega)$-functions with
compact support in $\Omega$. 

\begin{lemma}\label{Local_representation}
    Let $f\in W_p^\tau(\cM)$ based on the admissible augmented atlas
    $\cA_{\cM} =\{(U_j,\psi_j,\chi_j): 1\le j\le L\}$. Then, for
    all $1\leq j\leq L$ the local representation $f\circ \psi_j^{-1}$
    belongs to the local Sobolev space
    $W_{loc}^{\tau,p}(\psi_j(U_j))$. Moreover, for $\xi\in
    C_c^{\infty}(\psi_j(U_j))$ we can estimate 
    \[
    \|\xi(f\circ \psi_j^{-1})\|_{W^{\tau}_p(\psi_j(U_j))}\leq
    C\|f\|_{W_p^\tau(\cM)},
    \]
    where the implicit constant may depend on $\xi$.
\end{lemma}

Additionally, we will need the following equivalent characterization of local Sobolev functions.
\begin{lemma}\label{Characterization_locally_Sobolev}
    Let $\Omega\subseteq\R^d$ be an open, bounded Lipschitz domain,
    $1<p<\infty$ and $\tau\geq 0$. 
    Then, $f\in D'(\Omega)$ is in $W_{loc}^{\tau,p}(\Omega)$ if and only
    if for every precompact open set $V$ with
    $\overline{V}\subseteq\Omega$ there is a $w\in W_p^\tau(\Omega)$ such
    that $w|_V=f|_V$. 
\end{lemma}
    A proof for this can be found in \cite[Theorem 100]{Behzadan_loc}. For
    later, we note that the if direction can simply be established by
    setting $w=\Gamma f$ with a function $\Gamma\in
    C_{c}^{\infty}(\Omega)$ satisfying $\Gamma\equiv 1$ on an open
    neighborhood containing $\overline{V}$.

%%%%%%%%%%%%%%%%%%%%%%%%%%%%%%%%%%%%%%%%%%%%%%%
\subsection{Kernel-based approximation of functions defined  on
  manifolds}\label{Kernels_on_M}   
Compared to the situation in Section \ref{Klassisch_Skalar}, we are now
interested in domains $\Omega=\cM$, where $\cM$ is a smooth, compact and connected
$m$-dimensional embedded submanifold of $\R^{d}$. Of course, 
the operators $I_X$ and $Q_{X,\lambda}$ from Definition
\ref{approxOperators}  are well-defined if $\Omega=\cM$ and now become
operators from $C(\cM)$   to $V_X$ and Lemma \ref{lem:kernelsolution}
still holds. However, we now have 
to specify what kind of  positive definite kernels on 
$\cM$ we want to use. Although there are different ways to construct
kernels, particularly for specific manifolds (see for example \cite{FW20, FW29,
  FW24, FW14}), we will focus on kernels which arise by restricting positive
definite kernels defined on the ambient space $\R^{d}$ to the embedded submanifold
$\cM$, following particularly \cite{Fuselier-Wright-12-1}.

Consider a positive definite kernel $K:\R^{d}\times \R^{d}\to \R$ of
the form $K(x,y)=\Phi(x-y)=\phi(\|x-y\|_2)$. Restricting it to the compact
submanifold $\cM\subseteq\R^{d}$, gives the kernel  
\[
\Psi(\cdot,\cdot):=K(\cdot,\cdot)|_{\cM\times \cM}:\cM\times \cM\to
\R,
\]
which is obviously positive definite on $\cM$. For our error analysis
we need to know the corresponding reproducing kernel Hilbert
space. Fortunately, it turns out that if we start with a kernel $K$ of
this form which is the reproducing kernel of a Sobolev space then the
restricted kernel $\Psi$ will also be the reproducing kernel of a Sobolev
space on the manifold  $\cM$. To be more precise, the corresponding result to
Proposition \ref{locns} comes from  
\cite[Theorem 5]{Fuselier-Wright-12-1}.

\begin{proposition}\label{Native_on_M}
  Let $\tau>d/2$, $m\in\N$ be given and set $\sigma = \tau-(d-m)/2$.
  Assume $\Phi_\tau\in L_1(\R^d)\cap C(\R^d)$ has a Fourier
  transform $\widehat{\Phi_\tau}$ satisfying (\ref{ftdecay}). Let
  $\cM\subseteq\R^d$ be a compact, connected and 
  smooth $m$-dimensional submanifold.  Let $\Psi_\sigma:\cM\times\cM\to\R$ be
  defined by $\Psi_\sigma(x,y)=\Phi_\tau(x-y)$, $x,y\in\cM$. Then, there
  exists an inner product $\langle
  \cdot,\cdot\rangle_{\Psi_\sigma}:H^\sigma(\cM)\times
  H^\sigma(\cM)\to\R$ on  $H^\sigma(\cM)$ such that
  $\Psi_\sigma$ is the reproducing
  kernel of $H^\sigma(\cM)$ with respect to this inner product. The
  norm $\|\cdot\|_{\Psi_\sigma}$ induced by this inner product is
  equivalent to the  norm on $H^\sigma(\cM)$, i.e. there
  are constants $C_1,C_2>0$ such that
  \[
  C_1 \|u\|_{\Psi_\sigma} \le \|u\|_{H^\sigma(\cM)} \le C_2
  \|u\|_{\Psi_\sigma}, \qquad u\in H^\sigma(\cM).
  \]
  
\end{proposition}
Next, we analyze the approximation processes in this setting.
%\subsection{Errors for interpolation with restricted kernels}
%Next, we study the interpolation and penalized 
For $X=\{x_1,\ldots,x_N\}\subseteq\cM$ let
$V_X=\spn\{K(\cdot,x_1),\ldots,K(\cdot,x_N)\}=
\spn\{\Psi(\cdot,x_1),\ldots, \Psi(\cdot,x_N)\}$, using the notation above.
We now want to collect and prove error estimates for the interpolation
operator $I_X:C(\cM)\to V_X$ and the approximation operator
$Q_{X,\lambda}:C(\cM)\to V_X$. To formulate them, we
need to introduce the fill distance, the separation radius and their ratio
for a set $X=\{x_1,\ldots,x_N\}\subseteq \cM$ in the
manifold setting by
\[
  h_{X,\cM}:=\sup_{x\in \cM}\min_{x_i\in X} d_\cM(x,x_i),
  \quad q_X:=\frac{1}{2}\min_{x_i\ne x_j} d_{\cM}(x_i,x_j), \quad
  \rho_{X,\cM}:=h_{X,\cM}/q_{X}.
%  \label{Fill_distance_M}
\]
We obviously have $q_X\le h_{X,\cM}$. It will be our understanding
that we will use the intrinsic metric $d_\cM$ for these quantities if
the data points are on $\cM$ and the Euclidean distance otherwise. As
usual, we will call a data 
set $X\subseteq \cM$ quasi-uniform if there is a constant $c_{qu}\ge 1$
close to one such that $h_{X,\cM}\le c_{qu} q_X$. Obviously, such a
constant always exists but can become arbitrarily large. In the
context of approximation results and sampling inequalities, the concept
of quasi-uniform sets is best reflected if one thinks not about only one
data set $X$ but a sequence of sets $(X_j)$ with $h_{X_j,\cM}\le
c_{qu} q_{X_j}$ for all $j$.

While the interpolation operator was
comprehensively investigated in \cite{Fuselier-Wright-12-1}, nothing seems
to be known for the penalized least-squares operator in this context.
For the interpolation operator, Corollary 13 and Theorem 17 in
\cite{Fuselier-Wright-12-1} yield the following bounds. Note that
these bounds were derived by employing sampling inequalities and that we have
adjusted the range of $\mu$ according to the more general sampling
inequalities cited in Theorem \ref{Sampling_Ineq}.

\begin{theorem}\label{Fuselier-Wright-Final}
 Let $\Phi\in C(\R^{d})\cap L_1(\R^{d})$, satisfying  \eqref{ftdecay}
 with $\tau>d/2$. Let $\sigma=\tau-(d-m)/2$ and $1\leq q\leq \infty$.
 Then, there are constants $h_0>0$ and $C>0$ such that for all
 $X\subseteq\cM$ with $h_{X,\cM}<h_0$ and all  $f\in H^\beta(\cM)$ with
 $\beta>m/2$ the
 error between $f$ and its interpolant $I_Xf$ using the restricted
 kernel $\Psi$ can be bounded as follows.
\begin{enumerate}
\item   If $\beta\ge \sigma$ then
\[
\|f-I_Xf\|_{W_q^{\mu}(\cM)}\leq Ch_{X,\cM}^{\sigma-\mu-m(1/2-1/q)_+}\|f\|_{H^\sigma(\cM)}
\]
for all $0\le \mu\le \ell(\sigma,m,q)$ with $\mu\in\N_0$ if $q=\infty$.
\item If $\beta<\sigma$ then
\[
\|f-I_Xf\|_{W_q^{\mu}(\cM)}\leq
Ch_{X,\cM}^{\beta-\mu-m(1/2-1/q)_+}\rho_{X,\cM}^{\sigma-\beta}\|f\|_{H^{\beta}(\cM)}
\]
for all $0\le \mu\le \ell(\beta,m,q)$ with $\mu\in\N_0$ if $q=\infty$.
\end{enumerate}
\end{theorem}

To derive corresponding estimates for the penalized least-squares
operator, we first need to carry the sampling inequality from Theorem
\ref{Sampling_Ineq} over to the manifold setting. For this, 
we need the following result on the connection between local fill
distances and global ones.

\begin{lemma}\label{lem:fillequivalence}
  Let $\cM\subseteq \R^d$ be a smooth, connected and
  compact $m$-dimensional submanifold with admissible atlas $\cA_{\cM} =
  \{(U_j,\psi_j) : 1\le j\le 
  L\}$ with charts $\psi_j:U_j\subseteq\cM\to
  B_j:=\psi_j(U_j)\subseteq\R^m$. Set $X_j:=X\cap U_j$ and
  $Y_j:=\psi_j(X_j)$
  for $1\le j\le N$. Then, there are constants
  $c_1,c_2>0$, depending on the atlas such that
\[
  c_1 h_{Y_j,B_j} \le h_{X_j, U_j} \le c_2 h_{Y_j,B_j}, \qquad 1\le
  j\le L.
  \]
  Moreover, the maximum of the local fill distances is equivalent to
  the global one, i.e.
  \[
  c_1 h_{X,\cM} \le \max_{1\le j\le L} h_{X_j,U_j}\le c_2 h_{X,\cM},
  \]
where the last upper bound only holds for sufficiently small
$h_{X,\cM}\le h_0$.
  Finally, using the mesh ration $\rho_{X,\cM} = h_{X,\cM}/q_X$, the
  fill distances of  $X_j$ and $Y_j$, respectively,  
  are all comparable to $h_{X,\cM}$,  if $h_{X,\cM}\le h_0$, i.e.
  \[
    \rho_{X,\cM}^{-1} h_{X,\cM} \le  h_{X_j,U_j}\le c_2 h_{X,\cM}, \qquad 1\le j\le
    L.
    \]

  \end{lemma}
\begin{proof}
The first two statements follow immediately from Proposition 7 and
Theorem 8 from \cite{Fuselier-Wright-12-1}. For the last statement we
note that the second statement already yields $h_{Xj,U_j} \le c_2
h_{X,\cM}$. For the lower bound we note 
$
 h_{X,\cM} =\rho_{X,\cM} q_X \le  \rho_{X,\cM} q_{X_j} \le \rho_{X,\cM} h_{X_j,U_j}
$
for $1\le j\le L$.
\end{proof}
With this at our hand, we can now formulate and prove the required
sampling inequality for functions from $H^\sigma(\cM)$.

\begin{theorem}\label{thm:samplingManifold} 
Let $\cM$ be a compact, connected, smooth $m$-dimensional submanifold
of $\R^d$ with  admissible augmented atlas $\cA_{\cM} = \{
(U_j,\psi_j,\chi_j) : 1\le j\le L\}$. Let $p,q\in [1,\infty]$ and
$\gamma:=\max\{2,p,q\}$. Let $\sigma>m/2$.  Then, there exist constants $h_0>0$
 and $C>0$ such that for all $X=\{x_1,\ldots,x_N\}\subseteq
\cM$ 
with $h_{X,\cM}\le h_0$ and all $u\in H^\sigma(\cM)$,
we have
\[
|u|_{W_q^\mu(\cM)} \le C\left(h_{X,\cM}^{\sigma-\mu-m(1/2-1/q)_+}
|u|_{H^\sigma(\cM)} + \rho_{X,\cM}^{(\mu-m/\gamma)_+}
h_{X,\cM}^{m/\gamma-\mu} \|u\|_{\ell_p(X)}\right)
\]
for all $0\le \mu\le \ell(\sigma,m,q)$ with $\mu\in\N_0$ if 
$q=\infty$. 
\end{theorem}

\begin{proof}
  As in the last lemma, we let $X_j:=X\cap U_j$, $B_j:=\psi_j(U_j)$
  and $Y_j:=\psi_j(X_j)$. As  $u\in H^\sigma(\cM)$, we have $(\chi_j u)\circ
\psi_j^{-1}\in H^\sigma(B_j)$ for each chart $\psi_j$, $1\le j\le L$, where
$B_j$ is an open ball in $\R^m$ and, hence, in particular a Lipschitz domain.
Thus, we can apply the sampling inequality from Theorem
\ref{Sampling_Ineq} locally on each 
$B_j$ to obtain 
\[
    \|(\chi_j u)\circ \psi_j^{-1}\|_{W_q^\mu(B_j)}
    \leq C_j\left( h_{Y_j,B_j}^{\sigma-\mu-m(1/2-1/q)_+}\|(\chi_j u)\circ
    \psi_j^{-1}\|_{H^\sigma(B_j)} + h_{Y_j,B_j}^{m/\gamma-\mu}\|(\chi_ju)\circ\psi_j^{-1}\|_{\ell_p(Y_j)}\right)
\]
provided that $h_{Y_j,B_j}<h_0(B_j)$. Continuing under the assumption
\begin{equation}\label{h0_Bed}
h_{Y_j,B_j}\le h_0:=\min_{1\le j\le L} h_0(B_j), \qquad 1\le j\le L,
\end{equation}
we can sum up  the contributions of the finitely many coordinate
domains, yielding
\begin{eqnarray*}
    \|u\|_{W_q^{\mu}(\cM)}&=&\sum_{j=1}^L \|(\chi_ju)\circ \psi_j^{-1}\|_{W_q^{\mu}(B_j)}\\
    &\leq& \sum_{j=1}^LC_j\left( h_{Y_j,B_j}^{\sigma-\mu-m(1/2-1/q)_+}\|(\chi_j u)\circ
    \psi_j^{-1}\|_{H^\sigma(B_j)} + h_{Y_j,B_j}^{m/\gamma-\mu}\|(\chi_ju)\circ\psi_j^{-1}\|_{\ell_p(Y_j)}\right)\\
    &\leq&
    C h_{X,\cM}^{\sigma-\mu-m(1/2-1/q)_+} \|u\|_{H^{\sigma}(\cM)} +
    C\sum_{j=1}^L h_{X_j,U_j}^{m/\gamma-\mu} \|\chi_j u\|_{\ell_p(X_j)},
    %C \rho_{X,\cM}^{\mu-m/\gamma}
    %h_{X,\cM}^{m/\gamma-\mu} \sum_{j=1}^L \|\chi_j u\|_{\ell_p(X_j)},
\end{eqnarray*}
using Lemma \ref{lem:fillequivalence}, which also shows that the
condition (\ref{h0_Bed}) can be rewritten as $h_{X,\cM}\le
h_0$ with a slightly modified $h_0$.  Lemma \ref{lem:fillequivalence} also shows
$h_{X_j,U_j}^{m/\gamma-\mu}\le C h_{X,\cM}^{m/\gamma-\mu}$ if
$m/\gamma-\mu\ge 0$ and $h_{X_j,U_j}^{m/\gamma-\mu} \le C \rho_{X,\cM}^{\mu-m/\gamma}
h_{X,\cM}^{m/\gamma-\mu}$ if $m/\gamma-\mu<0$, which both together can
be written as $h_{X_j,U_j}^{m/\gamma-\mu} \le
\rho_{X,\cM}^{(\mu-m/\gamma)_+}h_{X,\cM}^{m/\gamma-\mu}$.

 Finally, using  $0\le \chi_j(x)\le 1$ for $1\le j\le L$ and $x\in\cM$
 and that we have only a finite number of charts, the
sum of the discrete $\ell_p(X_j)$ norms can be bounded by a constant
times $\|u\|_{\ell_p(X)}$, which yields the stated result.
\end{proof}

Note that the discrete term in the sampling inequality above means the
following. If we have $m/\gamma-\mu\ge 0$ or if the data set is
quasi-uniform then the term simply becomes $h_{X,\cM}^{m/\gamma-\mu}
\|u\|_{\ell_p(X)}$. However, if $m/\gamma-\mu<0$ and the data set is
not quasi-uniform the term is actually given by
$q_{X}^{m/\gamma-\mu}\|u\|_{\ell_p(X)}$. This can also be avoided if
the local fill distances are all of the same size and comparable to $h_{X,\cM}$.

The next result deals with stability of the approximation process
$Q_{X,\lambda}:C(\cM)\to V_X$.  However, as we will later on need to
deal with mismeasured data, we will formulate and prove this in a more
general way. Note that the definition of $Q_{X,\lambda}$ does not
require a function $f\in C(\cM)$. It can also be defined using a vector
$y\in\R^N$, as the operator only needs the data $(x_j,f(x_j))$, $1\le j\le
N$, which can be replaced by data of the form $(x_j,y_j)$, $1\le
j\le N$. In this situation, we will also simply write $Q_{X,\lambda}y$ for
the penalized least-squares approximation using these data.

\begin{lemma}\label{lem:stab}
Let $X=\{x_1,\ldots,x_N\}\subseteq\cM$, $f\in H$, where $H$ is a
reproducing kernel Hilbert space of real-valued functions defined on
$\cM$. Then, for any $\lambda>0$ and $y\in\R^N$,
the penalized least-squares approximation $Q_{X,\lambda}y\in H$ satisfies
\begin{eqnarray*}
\|Q_{X,\lambda}y\|_H & \le & \frac{1}{\sqrt{\lambda}}\|f(X)-y\|_2 +
  \|f\|_H, \\
  \|y-Q_{X,\lambda} y\|_2\ &\le& \|f(X)-y\|_2 + \sqrt{\lambda} \|f\|_H.
\end{eqnarray*}

\end{lemma}
\begin{proof} As $Q_{X,\lambda}y$ minimizes
  $J_{\lambda,y}(s)=\|s(X)-y\|_2^2+\lambda\|s\|_H^2$ over all $s\in H$, we immediately have
  \[
  \max\{\lambda \|Q_{X,\lambda}y\|_H^2, \|y-Q_{X,\lambda}y\|_2^2\} 
  \le J_{\lambda,y} (Q_{X,\lambda}y) \le J_{\lambda,y}(f) =
  \|y-f(X)\|_2^2+\lambda \|f\|_H^2.
  \]
Taking the square root and using the monotonicity of the square root
yields the stated bounds.
\end{proof}

With all these auxiliary results at hand, we can now formulate error
estimates for the penalized least-squares operator. We will  do
this for mismatched data, i.e. instead of observing $y=f(X)$ for a
function $f:\cM\to\R$, we work with the observations $y$ which might only be an
approximation to $f(X)$. However, we will restrict ourselves to
matching smoothness, i.e. we will assume that the target function
belongs to $H^\sigma(\cM)$. Rougher target functions $f\in
H^\beta(\cM)$ can also be handled very similarly to interpolation but
the terms will become more involved, so that we omit the details
here. For the same reason, we will assume that our data set is
quasi-uniform. 

\begin{theorem}\label{thm:quasi-manifold-error}
Let $\Phi\in C(\R^{d})\cap L_1(\R^{d})$, satisfying  \eqref{ftdecay}
 with $\tau>d/2$. Let $\sigma=\tau-(d-m)/2$ and $1\leq q\leq \infty$.
 Then, there are constants $h_0>0$ and $C>0$ such that for all quasi-uniform
 $X\subseteq\cM$ with $h_{X,\cM}<h_0$ and all  $f\in H^\sigma(\cM)$ and
 all data $y\in\R^N$, observed at $X$, the
 error between $f$ and the penalized least-squares approximant
 $Q_{X,\lambda}y$ using the restricted 
 kernel $\Psi$ and $\lambda>0$ can be bounded as
  \begin{eqnarray*}
\|f-Q_{X,\lambda} y\|_{W_q^\mu(\cM)} &\le& C
\left(h_{X,\cM}^{\sigma-\mu-m(1/2-1/q)_+} + \sqrt{\lambda}h_{X,\cM}^{m/\gamma-\mu}\right)
\|f\|_{H^\sigma(\cM)}  \\
&&\mbox{} + C
  \left( h_{X,\cM}^{\sigma-\mu-m(1/2-1/q)_+} \frac{1}{\sqrt{\lambda}} +
  h_{X,\cM}^{m/\gamma-\mu}\right) \|f(X)-y\|_2,
  \end{eqnarray*}
for all $0\le \mu\le \ell(\sigma,m,q)$ with $\mu\in\N_0$ if
$q=\infty$.  In particular, if the observed data are given by
$y_j=f(x_j)$, $1\le j\le N$, this simplifies to
\[
\|f-Q_{X,\lambda} f\|_{W_q^\mu(\cM)} \le C
\left(h_{X,\cM}^{\sigma-\mu-m(1/2-1/q)_+} + \sqrt{\lambda} h_{X,\cM}^{m/\gamma-\mu}\right)
\|f\|_{H^\sigma(\cM)}.
\]
\end{theorem}

\begin{proof}
The sampling inequality from Theorem \ref{thm:samplingManifold} yields
for $u=f-Q_{X,\lambda}y$ the immediate bound
\[
\|f-Q_{X,\lambda} y\|_{W_q^\mu(\cM)} \le C\left(
h_{X,\cM}^{\sigma-\mu-m(1/2-1/q)_+}
  \|f-Q_{X,\lambda}y\|_{H^\sigma(\cM)} + h_{X,\cM}^{m/\gamma-\mu}
  \|f-Q_{X,\lambda} y\|_{\ell_2(X)}\right).
  \]
  Next, using Lemma \ref{lem:stab}, we find
  \[
  \|f- Q_{X,\lambda} y\|_{H^\sigma(\cM)} \le \|f\|_{H^\sigma(\cM)} +
  \|Q_{X,\lambda} y\|_{H^\sigma(\cM)} \le 2 \|f\|_{H^\sigma(\cM)} +
  \frac{1}{\sqrt{\lambda}} \|f(X)-y\|_2
  \]
  and, similarly, 
  \[
  \|f-Q_{X,\lambda}y\|_{\ell_2(X)} \le 2 \|f(X)-y\|_2 +
  \sqrt{\lambda}\|f\|_{H^\sigma(\cM)}.
  \]
Plugging both of these estimates into the bound above yields the
desired error estimate.  
\end{proof}
 
%%%%%%%%%%%%%%%%%%%%%%%%%%%%%%%%%%%%%%%%%%%%%%%%%
%\subsection{Vector-valued interpolation on manifolds}\label{Vector_valued_M}
Until now we have discussed the approximation of functions $f:\cM\to
\R$. To reuse this theory for approximating functions $f:\cM\to \cN$,
we briefly need to discuss the approximation of vector-valued functions
$f:\cM\to \R^{\ell}$.  Extending kernel-based methods to
non-scalar codomains has already thoroughly been investigated for many
years, see for example \cite{ VVK1, VVK2, VVK3, GieslWendland}.
In the most general setting, one considers functions $f:\Omega\to W$, where $W$ is some
Hilbert space. However, in our situation, where $W=\R^{\ell}$, we can
also simply use component-wise approximation, which corresponds to
employing a diagonal kernel in the general theory. Hence, we define
the vector-valued interpolation and approximation operators as
follows. For $f=(f_1,\ldots,f_{\ell}):\cM\to\R^{\ell}$ we define
$I_{X,\R^{\ell}}f:\cM\to\R^{\ell}$ by $(I_{X,\R^{\ell}}f)_j = I_Xf_j$  and
$Q_{X,\lambda,\R^{\ell}}f:\cM\to\R^{\ell}$ by $(Q_{X,\lambda,\R^{\ell}}f)_j=Q_{X,\lambda}f_j$ for
$1\le j\le \ell$, using the same restricted kernel
$\Psi:\cM\times\cM\to\R$ for each component. In both cases, these operators do not actually
require the data to be generated by a function $f:\cM\to\R^{\ell}$, it
suffices to have observations $Y=(y_1,\ldots,y_{\ell})$ with
$y_j\in\R^N$ to make them well-defined. As before, we will only use
this in the case of the penalized least-squares operator.

To carry our estimates over to vector-valued Sobolev spaces, we
formally define these as follows.

\begin{definition}\label{Vector_valued_Sobolev}
    Let $\cM\subseteq\R^{d}$ be an embedded submanifold. For $1\leq q\leq\infty$ 
and $\mu\geq 0$, where $\mu\in\N_0$ if $q=\infty$, we define the
    $\R^{\ell}$-valued Sobolev space  
$W_q^{\mu}(\cM,\R^{\ell}):=\{f=(f_1,\dots,f_{\ell}): f_i\in W_q^{\mu}(\cM), 1\leq i \leq \ell\}$
with norm
\begin{equation}
  \|f\|_{W_q^{\mu}(\cM,\R^{\ell})}:=\left(\sum_{i=1}^{\ell}
  \|f_i\|_{W_q^{\mu}(\cM)}^2\right)^{\frac{1}{2}}.\label{Vector_Norm} 
\end{equation}
\end{definition}
 Obviously, we could equivalently use any other norm on $\R^{\ell}$ for
 combining the component-wise contributions in
 \eqref{Vector_Norm}. For later, we emphasize the inequality  
\begin{equation}
    \esssup_{x\in \cM} \|f(x)\|_2\leq \|f\|_{L_{\infty}(\cM,\R^{\ell})},\quad f\in L_{\infty}(\cM,\R^{\ell}),\label{unendlich_Ungleichung}
\end{equation}
which is straightforward to verify from the definitions.

For three vectors $a,b,c\in\R^{\ell}$ and two
scalars $\alpha,\beta>0$, which satisfy the inequalities $|a_i|\le \alpha
|b_i| + \beta |c_i|$ for $1\le i\le \ell$, we have the bound
$
  \|a\|_2 \le \alpha\|b\|_2 + \beta\|c\|_2.
  $
Moreover, the norm-minimal property of the interpolant and the
minimizing property of the penalized least-squares approximant show,
that these properties remain valid in the vector-valued setting,
i.e. we have on the one hand
\begin{equation}\label{stabvectorinterpolant} 
\|I_{X,\R^{\ell}}f\|_{H^{\sigma}(\cM,\R^{\ell})} \le
\|f\|_{H^\sigma(\cM,\R^{\ell})}
\end{equation}
and, using also $\sum_{i=1}^\ell
\|y_i-f_i(X)\|_2^2 = \|Y-f(X)\|_2^2$ for $Y=(y_1,\ldots,
y_{\ell})$ with $y_i\in\R^N$, on the other hand
\begin{eqnarray}
\|Q_{X,\lambda,\R^{\ell}}Y\|_{H^\sigma(\cM,\R^{\ell})} & \le & \frac{1}{\sqrt{\lambda}}\|f(X)-Y\|_2 +
  \|f\|_{H^\sigma(\cM,\R^{\ell})},\label{qxl1} \\
  \|Y-Q_{X,\lambda} Y\|_2\ &\le& \|f(X)-Y\|_2 + \sqrt{\lambda} \|f\|_{H^\sigma(\cM,\R^{\ell})}.\nonumber
\end{eqnarray}

This particularly shows that both Theorem \ref{Fuselier-Wright-Final} for the
interpolation operator, as well as Theorem
\ref{thm:quasi-manifold-error} for the penalized least-squares
operator carry over to the vector-valued
setting.

\begin{corollary}\label{Error_M_Vektorwertig}
 Let $\Phi\in C(\R^{d})\cap L_1(\R^{d})$, satisfying  \eqref{ftdecay}
 with $\tau>d/2$. Let $\sigma=\tau-(d-m)/2$ and $1\leq q\leq \infty$.
 Then, there are constants $h_0>0$ and $C>0$ such that for all
 $X\subseteq\cM$ with $h_{X,\cM}<h_0$ and all  $f\in H^\beta(\cM,\R^{\ell})$ with
 $\beta>m/2$ the
 error between $f$ and its interpolant $I_{X,\R^{\ell}}f$ using the restricted
 kernel $\Psi$ can be bounded as follows.
\begin{enumerate}
\item   If $\beta\ge \sigma$ then
\[
\|f-I_{X,\R^{\ell}}f\|_{W_q^{\mu}(\cM,\R^{\ell})}\leq
Ch_{X,\cM}^{\sigma-\mu-m(1/2-1/q)_+}\|f\|_{H^\sigma(\cM,\R^{\ell})}, 
\]
for all $0\le \mu\le \ell(\sigma,m,q)$ with $\mu\in\N_0$ if $q=\infty$.
\item If $\beta<\sigma$ then
\[
\|f-I_{X,\R^{\ell}}f\|_{W_q^{\mu}(\cM,\R^{\ell})}\leq
Ch_{X,\cM}^{\beta-\mu-m(1/2-1/q)_+}\rho_{X,\cM}^{\sigma-\beta}\|f\|_{H^{\beta}(\cM,\R^{\ell})}
\]
for all $0\le \mu\le \ell(\beta,m,q)$ with $\mu\in\N_0$ if $q=\infty$.
\end{enumerate}
\end{corollary}

\begin{corollary}\label{Error-quasi-M-vector}
Let $\Phi\in C(\R^{d})\cap L_1(\R^{d})$, satisfying  \eqref{ftdecay}
 with $\tau>d/2$. Let $\sigma=\tau-(d-m)/2$ and $1\leq q\leq \infty$.
 Then, there are constants $h_0>0$ and $C>0$ such that for all
 $X\subseteq\cM$ with $h_{X,\cM}<h_0$ and all  $f\in H^\sigma(\cM,\R^{\ell})$ and
 all at $X$ observed data $Y=(y_1,\ldots,y_{\ell})$ with
 $y_j\in\R^N$, $1\le j\le \ell$, the 
 error between $f$ and the penalized least-squares approximant
 $Q_{X,\lambda,\R^{\ell}}Y$ using the restricted 
 kernel $\Psi$ and $\lambda>0$ can be bounded as
  \begin{eqnarray*}
\|f-Q_{X,\lambda,\R^{\ell}} Y\|_{W_q^\mu(\cM,\R^{\ell})} &\le& C
\left(h_{X,\cM}^{\sigma-\mu-m(1/2-1/q)_+} + \sqrt{\lambda}h_{X,\cM}^{m/\gamma-\mu}\right)
\|f\|_{H^\sigma(\cM,\R^{\ell})}  \\
&&\mbox{} + C
  \left( h_{X,\cM}^{\sigma-\mu-m(1/2-1/q)_+} \frac{1}{\sqrt{\lambda}} +
  h_{X,\cM}^{m/\gamma-\mu}\right) \|f(X)-Y\|_2,
  \end{eqnarray*}
for all $0\le \mu\le \ell(\sigma,m,q)$ with $\mu\in\N_0$ if
$q=\infty$.  In particular, if the observed data are given by
$y_j=f(x_j)$, $1\le j\le N$, this simplifies to
\[
\|f-Q_{X,\lambda,\R^{\ell}} f\|_{W_q^\mu(\cM,\R^{\ell})} \le C
\left(h_{X,\cM}^{\sigma-\mu-m(1/2-1/q)_+} + \sqrt{\lambda}h_{X,\cM}^{m/\gamma-\mu}\right)
\|f\|_{H^\sigma(\cM,\R^{\ell})}.
\]
\end{corollary}

%%%%%%%%%%%%%%%%%%%%%%%%%%%%%%%%%%%%%%%%%%%%%%%%
\section{Learning of manifold-to-manifold maps}\label{Together_Section}
Throughout this section, $\cM\subseteq\R^d$ will be a compact, smooth, connected
$m$-dimensional submanifold of $\R^d$ and $\cN\subseteq\R^{\ell}$ will be a compact smooth,
connected $n$-dimensional submanifold of $\R^{\ell}$.

After having established theory for approximating functions $f:\cM\to
\R^{\ell}$ from scattered data, we can now come to our main task, the 
approximation of manifold-valued functions $f:\cM\to \cN$. We
will achieve this by combining the approximation 
processes of the last section with a projection $\cP_\cN:\R^{\ell}\to\cN$. Hence,
in this section, we will study the combined interpolation and approximation operators
\[
    I_{X,\cN}:= \cP_{\cN}\circ I_{X,\R^{\ell}},\qquad
    Q_{X,\lambda,\cN}:=\cP_{\cN}\circ
    Q_{X,\lambda,\R^{\ell}}. 
\]

An obvious choice for $\cP_\cN$ is to assign each vector $u\in
\R^{\ell}$ to its closest point on $\cN$. This however, is usually only
possible if $u$ is already sufficiently close to the manifold $\cN$. 

\subsection{Closest point approximation} \label{Tubular_Neighborhood_Section} 
In this subsection, we will recall and prove relevant material on
the closest point projection $\cP_\cN:\R^{\ell}\to\cN$. In particular,
we are interested in its well-definedness.

Let $u\in \R^{\ell}$. A point  $p^*\in \cN$ is called a closest
    point to $u$ from $\cN$ if $p^*=\argmin_{p\in \cN} \|u-p\|_2.$ 
Since $\cN$ is assumed to be compact, a closest point to $u\in
\R^{\ell}$ always exist. However, it does not have to be
unique.
%Consider for example the sphere $S^{\ell-1}\subseteq\R^{\ell}$,
%where each point $p\in S^{\ell-1}$ is a closest point to the origin
%$0\in \R^{\ell}$.
As in the situation of best approximations from
linear subspaces, closest points are characterized by some type of
orthogonality. 
\begin{lemma}\label{Orthogonality of best approximation}
Let $p\in \cN$ be a closest point to $u\in \R^{\ell}$. Then $u-p\in N_p\cN.$
\end{lemma}
\begin{proof}
    Let $v\in T_p\cN$. Consider an arbitrary curve
    $c:(-\epsilon,\epsilon)\to \cN$ with $c(0)=p$ and
    $\dot{c}(0)=v.$ As $p$ is a closest point to $u$, the function 
    $g:(-\epsilon,\epsilon)\to \R, t\mapsto\|u-c(t)\|_2^2$
    has a local minimum at $t=0$. Hence,
$
0 = \dot{g}(0)=\frac{d}{dt}|_{t=0}\langle u-c(t),u-c(t)\rangle_2
 = 2\langle c(0),\dot{c}(0)\rangle_2- 2\langle u,\dot{c}(0)\rangle_2=
 2\langle u-p,v\rangle_2
 $ 
     and therefore $u-p\perp v$.
\end{proof}
Taking a closer look at the unit sphere
$S^{\ell-1}\subseteq\R^{\ell}$, we note that for each $u\in
\R^{\ell}\setminus{\{0\}}$ there exists a unique closest point
from $S^{\ell-1},$ which is given by $u\mapsto u/\|u\|_2$. Hence, in
this case there exists an open neighborhood $U$ of $S^{\ell-1}$, such
that each $u\in U$ has a unique closest point from
$S^{\ell-1}$. This is no particularity of the
sphere, but holds in general. Such a $U$ is called  a
tubular neighborhood, see  \cite[Chapter 6]{LeeSmooth}. 
\begin{definition}
    A tubular neighborhood of an embedded submanifold
    $\cN\subseteq\R^{\ell}$ is a neighborhood $U\subseteq\R^{\ell}$ of
    $\cN$ which is the diffeomorphic image under the addition map $E:
    N\cN\to \R^{\ell}, (p,n)\mapsto p+n$, of an open subset $V\subseteq
    N\cN$ of the form
$V=\{(p,n)\in N\cN: \|n\|_2<\delta(p)\}$
for some positive continuous function $\delta:\cN\to \R.$ If the
function $\delta$ can be chosen as constant, $\delta\equiv
\hat{\delta}>0$, we say that $U=E(V)$ is a uniform tubular
neighborhood of $\cN$. 
\end{definition}

The tubular neighborhood theorem,  \cite[Theorem 6.24]{LeeSmooth}),
guarantees that every embedded submanifold of $\R^{\ell}$ has a tubular neighborhood.

\begin{theorem}[Tubular Neighborhood Theorem]\label{TubularNeighborhoodTheorem}
Every embedded submanifold $\cN$ of $\R^{\ell}$ has a tubular
neighborhood. If $\cN$ is compact, there also exists a uniform tubular
neighborhood of $\cN$. 
\end{theorem}
\begin{proof}
The first statement is the above mentioned \cite[Theorem 6.24]{LeeSmooth}),
The second statement follows easily: If $\cN$ is compact, the positive and
continuous function $\cN \ni p\mapsto \frac{1}{2}\delta(p)\in \R$
attains a minimum on $\cN$, denoted by $\hat{\delta}$. Then, 
$U_{\hat{\delta}}:= E(\{(p,n)\in N\cN: \|n\|_2<\hat{\delta}\})$
is a uniform tubular neighborhood of $\cN$.  
\end{proof}

Uniform tubular neighborhoods and closest points are
closely related.

\begin{theorem}\label{BestApproximationTubular}
Let $U=E(V)$ be a uniform tubular neighborhood of $\cN$ with
 $V=\{(p,n)\in N\cN: \|n\|_2<\hat{\delta}\}$
 for a $\hat{\delta}>0.$ Then, for each $u\in U$ there exists a unique
 closest point from $\cN$. The map  
 $$\cP_{\cN}: U\to \cN, \quad \cP_\cN(u):= \argmin_{p\in \cN} \|u-p\|_2,$$
 which assigns each point $u\in U$ to its closest point on $\cN$, is given by
 $\cP_{\cN}(u)=\pi_{N\cN}\circ E^{-1},$
where
$\pi_{N\cN}:N\cN\to \cN$ denotes the canonical projection from $N\cN$
to $\cN$. The map $\cP_{\cN}$ is smooth and is called closest point
projection. 
\end{theorem}
\begin{proof}
Let $u\in U$. As already mentioned, compactness of $\cN$ guarantees
the existence of a closest point to $u$ on $\cN$.
As $u\in U$, it can be uniquely written as $u=E((p,n))$ with $p\in
\cN$, $n\in T_p\cN$ and $\|n\|_2<\hat{\delta}$. Assume a point $p_2\in
\cN$ would also be a closest point to $u$. Then, we must have
$\|u-p_2\|_2\leq \|n\|_2<\hat{\delta}$. On the other hand, $u-p_2\in
N_{p_2}\cN$ according to Lemma \ref{Orthogonality of best
  approximation}. Hence, $(p_2,u-p_2)\in V$ with
$u=E((p_2,u-p_2))$. The uniqueness of such a decomposition inside the
tubular neighborhood $U$ gives $p=p_2$. We immediately conclude that
the closest point to $u\in U$ is given by  
$\cP_{\cN}(u)=\pi_{N\cN}\circ E^{-1}.$
As $E$ is a diffeomorphism and $\pi_{N\cN}$ is smooth, the closest
point projection $\cP_{\cN}:U\subseteq\R^{\ell}\to\cN$ is smooth.  
\end{proof}

We end this subsection by looking at the smoothness of the closest
point projection, when applied to a function $f\in
H^\sigma(\cM,\R^{\ell})$, i.e. we are interested in the smoothness of
$\cP_{\cN}\circ f:\cM\to\cN$. To this end, we need the following
result on the composition of functions.

\begin{lemma}\label{Nemitskij_result}
    Let $\Omega\subseteq\R^d$ be a bounded Lipschitz domain and
    $\sigma>\max\{1,d/2\}$. %d/2 für Sobolev Embedding, \geq 1 für
                       %Runst/Sickel-Resultat; Compact support will
                       %man dann hier wegen unten "Shift-Argument"
                       %unten auch nicht voraussetzen, ist aber auch
                       %gar keine Voraussetzung. 
Let $g\in C^{\infty}(\R^{\ell})$ with $g(0)=0$. 
Then, there is a constant $c>0$ such that 
\[
\|g\circ f\|_{H^\sigma(\Omega)}\leq c\left( \max_{1\leq j\leq
  \ell}\|f_j\|_{H^\sigma(\Omega)}\right)\left( 1+\max_{1\leq j\leq
  \ell}\|f_j\|^{\lfloor \sigma\rfloor}_{H^\sigma(\Omega)}\right)
\]
for all $f=(f_1,\ldots,f_\ell):\Omega\to\R^{\ell}$ with $f_j\in H^\sigma(\Omega)$ for $1\le
j\le \ell$.
\end{lemma} 
\begin{proof}
This lemma is a direct application of Theorem 2 from \cite[Section
  5.5.1]{Nemytskij}. To see this, first note that by \cite[Section
  2.1.2]{Nemytskij}, the Sobolev space $H^\sigma(\R^d)$ coincides with the
Triebel-Lizorkin space $F_{2,2}^\sigma=\F_{2,2}^\sigma$ for any $\sigma>0$ with
equivalent norms. Next, besides the above assumptions on $g$, Theorem 2 in
\cite{Nemytskij} also requires $f\in H^\sigma(\R^d)\cap L_\infty(\R^d)$,
which is satisfied because of $\sigma>d/2$ and the Sobolev embedding
theorem. Finally, in our situation the requirements on $\sigma$ are
$1\le \sigma$ and $0<\sigma<\mu$, which we satisfy with
$\mu=\lfloor\sigma\rfloor+1$. Thus, the  mentioned theorem
guarantees the existence of a constant $c>0$ such that
\[
\|g\circ
f\|_{H^\sigma(\R^d)}\leq c\left( \max_{1\leq j\leq
  \ell}\|f_j\|_{H^\sigma(\R^d)}\right)\left( 1+\max_{1\leq j\leq
  \ell}\|f_j\|^{\lfloor \sigma \rfloor}_{H^\sigma(\R^d)}\right)
\]
for all $f:\R^d\to \R^{\ell}$ with $f_j\in H^\sigma(\R^d)$,
$1\le j\le \ell$, where we also have used the Sobolev embedding
theorem, which allowed us to bound the $\|f_j\|_{L_\infty(\R^d)}$ norm
by a constant times $\|f_j\|_{H^\sigma(\R^d)}$.

As $\Omega\subseteq\R^d$ is a Lipschitz domain, the
functions $f_j\in H^\sigma(\Omega)$ can be extended to functions
$\tilde{E}f_j\in H^\sigma(\R^d)$ with $\|\tilde{E}f_j\|_{H^\sigma(\R^d)}\leq
C\|f_j\|_{H^\sigma(\Omega)}$. Noticing that $\|g\circ f\|_{H^\sigma(\Omega)}\leq
\|g\circ \tilde{E}f\|_{H^\sigma(\R^d)}$ finalizes our proof. 
\end{proof}

The application of the above lemma to the coordinate representation of
$\cP_\cN\circ f$ gives the desired result on the smoothness of the
latter function and Sobolev bounds thereof.

\begin{theorem}\label{thm:ProjectionNormEstimate}
Let $U\subseteq\R^{\ell}$ be a uniform tubular neighborhood  of
$\cN$ with associated closest point projection $\cP_{\cN}:U\to
\cN\subseteq\R^{\ell}$. Then, for any $\sigma>\max\{1,m/2\}$ there is
a constant $C>0$ such  that
\[
\|\cP_{\cN}\circ f\|_{H^\sigma(\cM,\R^{\ell})} \le C
\max\left\{1,\|f\|_{H^\sigma(\cM,\R^{\ell})}\left(1+\|f\|_{H^\sigma(\cM,\R^{\ell})}^{\lfloor\sigma\rfloor}\right)\right\}
\]
for all $f\in H^\sigma(\cM,\R^{\ell})$ with $f(\cM)\subseteq U$.
\end{theorem}
\begin{proof}
  By potentially shrinking the tubular neighborhood $U$, we obtain a
  uniform tubular neighborhood $\widetilde{U}$ of $\cN$ and an
  associated closest point projection
  $\widetilde{\cP}_\cN:\widetilde{U}\to\cN$, which has a smooth
  extension to the boundary of $\widetilde{U}$. By the extension lemma for smooth functions,
  \cite[Lemma 2.26]{LeeSmooth},  we 
can extend $\cP_{\cN}$ to all of $\R^{\ell}$ such that the
extension $G=(g_1,\dots, g_{\ell}):\R^{\ell}\to \R^{\ell}$ is smooth,
coincides with $\cP_{\cN}$ on $\widetilde{U}$ and has compact support
contained in $U$. 
By definition we have
\[
\|\cP_{\cN}\circ f\|^2_{H^\sigma(\cM,\R^{\ell})}
=\sum_{j=1}^{\ell}\|(\cP_{\cN}\circ f)_j\|^2_{H^\sigma(\cM)}
=\sum_{j=1}^{\ell}\left(\sum_{i=1}^L\|(\chi_i\circ
\psi_i^{-1})(g_j\circ f\circ
\psi_i^{-1})\|_{H^\sigma(\psi_i(U_i))}\right)^2. 
\]
Thus, we can restrict ourselves to estimating terms of the form
$\|(\chi_i\circ \psi_i^{-1})(g_j\circ f\circ \psi_i^{-1})\|_{H^\sigma(\psi_i(U_i))}$. As $f\in
H^\sigma(\cM,\R^{\ell})$, each component $f_l\in H^\sigma(\cM)$,
$1\le l\le \ell$, and hence $(f\circ \psi_i^{-1})_l\in
H_{loc}^\sigma(\psi_i(U_i))$ by Lemma \ref{Local_representation}. As 
$\text{supp}(\chi_i\circ \psi_i^{-1})\subseteq\psi_i(U_i)$ is compact,
there exists an open, precompact set $V$ with $\text{supp}(\chi_i\circ
\psi_i^{-1})\subseteq
V\subseteq\overline{V}\subseteq\psi_i(U_i)$. According to Lemma
\ref{Characterization_locally_Sobolev}, we can choose $\Gamma\in
C_c^{\infty}(\psi_i(U_i))$ with $\Gamma\equiv 1$ on an open
neighborhood containing $\overline{V}$ and get that $w_l:=\Gamma\cdot
(f\circ \psi_i^{-1})_l\in H^\sigma (\psi_i(U_i))$ with
$w_l|_{\overline{V}}=(f\circ \psi_i^{-1})_l|_{\overline{V}}$. We write
$w=(w_1,\dots,w_{\ell}):\psi_i(U_i)=B_i\subseteq\R^m\to \R^{\ell}.$ Exploiting that
multiplication with the smooth and compactly supported function
$\chi_i\circ \psi_i^{-1}\in C_c^{\infty}(\psi_i(U_i))$ is a continuous
linear operator on $H^\sigma(\psi_i(U_i))$ (see, e.g., \cite[Corollary
  3]{Behzadan_Holst}), we obtain  
\[
 \|(\chi_i\circ \psi_i^{-1})(g_j\circ f \circ 
 \psi_i^{-1})\|_{H^\sigma(\psi_i(U_i))}=\|(\chi_i\circ
 \psi_i^{-1})(g_j\circ w)\|_{H^\sigma(\psi_i(U_i))}\leq C\|g_j\circ
 w\|_{H^\sigma(\psi_i(U_i))}. 
\]
Setting $\widetilde{g}_j:=g_j-g_j(0):\R^{\ell}\to\R$, $1\le j\le \ell$, we can apply
Lemma \ref{Nemitskij_result} to derive
\begin{align*}
    \|g_j\circ w\|_{H^\sigma(\psi_i(U_i))}&\leq c\left( \max_{1\leq l\leq
      \ell}\|w_l\|_{H^\sigma(\psi_i(U_i))}\right)\left( 1+\max_{1\leq l\leq
      \ell}\|w_l\|^{\lfloor\sigma
      \rfloor}_{L_{\infty}(\psi_i(U_i))}\right)+\|g_j(0)\|_{H^\sigma(\psi_i(U_i))}\\  
    &\leq c\left( \max_{1\leq l\leq \ell}\|\Gamma\cdot(
     f\circ
    \psi_i^{-1})_l\|_{H^\sigma(\psi_i(U_i))}\right)\left( 1+\max_{1\leq
      l\leq \ell}\|\Gamma\cdot( f\circ
    \psi_i^{-1})_l\|^{\lfloor \sigma \rfloor}_{H^\sigma(\psi_i(U_i))}\right)\\ 
    &+|g_j(0)| \text{vol}(\psi_i(U_i))^{\frac{1}{2}}.
\end{align*}
As $g_j\in C_{c}^{\infty}(\R^{ \ell})$ and $\psi_i(U_i)=B_i\subseteq\R^m$ is bounded for
$1\le i\le L$, the last summand can be bounded by a constant
$c_{\cM}>0$ independent of $f$. Using Lemma \ref{Local_representation}
we proceed with  
\begin{align*} \|(\chi_i\circ \psi_i^{-1})(g_j\circ f\circ
  \psi_i^{-1})\|_{H^\sigma(\psi_i(U_i))}&\leq c\left( \max_{1\leq
    l\leq \ell}\|f_l\|_{H^\sigma(\cM)}\right)\left(
  1+\max_{1\leq l\leq \ell}\|f_l\|^{\lfloor    \sigma \rfloor}_{H^\sigma(\cM)}\right) + c_{\cM}\\ 
&\leq c \max\left\{1, \|f\|_{H^\sigma(\cM,\R^{\ell})}\left(
  1+ \|f\|^{\lfloor  \sigma
    \rfloor}_{H^\sigma(\cM,\R^{\ell})}\right)\right\}.
\end{align*} 
Overall, this gives
\[
\|\cP_{\cN}\circ f\|_{H^\sigma(\cM,\R^{\ell})}\leq c
\max\left\{1,\|f\|_{H^\sigma(\cM,\R^{\ell})}\left( 1+ \|f\|^{\lfloor
  \sigma\rfloor}_{H^\sigma(\cM,\R^{\ell})}\right)\right\}.
\]
\end{proof}

\begin{comment}
Let $\Phi$ and $s$ be chosen as in Corollary
\ref{Error_M_Vektorwertig}. Then, there exists an $\tilde{h}_2>0$ such
that for all $X\subseteq\cM$ with $h_{X,\cM}<\tilde{h}_2$ and $f\in
H^{s}(\cM,\cN)$, the resulting vector-valued interpolant
$I_{\R^{\ell}}f$ takes values in $U$. 
In this case, $I_{X,\cN}f\in H^{s}(\cM,\cN)$ is well-defined and 
\begin{align*}
    \|I_{\cN}f\|_{H^s(\cM,\R^{\ell})}\leq c
    \max\left\{1,\|f\|_{H^s(\cM,\R^{\ell})}\left( 1+ \|f\|^{\lfloor
      s\rfloor}_{H^s(\cM,\R^{\ell})}\right)\right\}, 
\end{align*}
\end{comment}

\subsection{Kernel-based learning combined with closest point
  projections}

We now come back to studying the interpolation and approximation
operators $I_{X,\cN}= \cP_{\cN}\circ I_{X,\R^{\ell}}$ and
$Q_{X,\lambda,\cN}=\cP_{\cN}\circ Q_{X,\lambda,\R^{\ell}}$. They are using a restricted kernel
$\Psi:\cM\times\cM\to\R$, $\Psi(x,y)=\Phi(x-y)$ for each
component. This kernel is supposed to be a 
reproducing kernel of $H^\sigma(\cM)$ with $\sigma>\max\{m/2,1\}$,
which is satisfied if $\Phi$ has a Fourier transform satisfying
(\ref{ftdecay}) with $\tau> \max\{d/2,(d-m)/2+1\}$ as we have
$\sigma=\tau-(d-m)/2$. The condition on $\tau$ is for example
satisfied if $\tau>d/2$ and $m\ge 2$ or if $\tau>(d+1)/2$.

To ensure our operators to be well-defined, we need that $I_{X,\R^{\ell}} f(x)$
and $Q_{X,\lambda,\R^{\ell}}f (x)$ belong to a uniform tubular neighborhood $U$
of $\cM$. We will use standard $L_\infty$-error estimates to achieve
this. To apply these estimates, we need the target function
$f:\cM\to\cN$ to belong to a Sobolev space $H^\beta(\cM,\R^{\ell})$. Thus, we aim to
derive error bounds for interpolating functions from the set 
\[
  H^{\beta}(\cM,\cN):=\{f\in H^\beta(\cM,\R^{\ell}): f(x)\in \cN
  \text{ almost everywhere}\}.
\]

Unfortunately, using standard $L_\infty$-error estimates to ensure that $I_{X,\R^{\ell}}f$ and
$Q_{X,\lambda,\R^{\ell}}f$ map $\cM$ into the tubular neighborhood $U$
of $\cN$ also means that the necessary fill distance will depend on the function $f$. One
way of avoiding this is to restrict ourselves to functions $f\in
H^{\beta}(\cM,\cN)$ satisfying $\|f\|_{H^\beta(\cM,\R^{\ell})} \le 1$,
where $1$ can be replaced by any other constant, which, of course,
only shifts the problem to identifying whether $f$ has a bounded norm
or not.

The above results now allow us to show that our interpolation operator
$I_{X,\cN}=\cP_{\cN}\circ I_{X,\R^\ell}$ and approximation operator
$Q_{X,\lambda,\cN}=\cP_{\cN}\circ Q_{X,\lambda,\R^\ell}$ are well-defined and
they allow us to derive error estimates for them. We start with the
interpolation operator.

\begin{theorem}\label{Final_estimate}
  Let $\tau>\max\{d/2, (d-m)/2+1\}$. Let $\sigma=\tau-(d-m)/2$,
  $\beta\ge \max\{m/2,1\}$ and $1\leq q\leq \infty$. Let $U\subseteq\R^{\ell}$ be a uniform
  tubular neighborhood of $\cN$.
  Let $\Phi\in C(\R^d)\cap L_1(\R^d)$
  satisfy (\ref{ftdecay}) and let $\Psi:\cM\times\cM\to\R$ be the
  restricted kernel. Then, the following statements for the
  interpolation operator  $I_{X,\cN}$ hold.
  \begin{enumerate}
  \item If $\beta\ge \sigma$ then there are constants $h_0,C>0$ such that for all
    $X\subseteq\cM$ with $h_{X,\cM}\le h_0$ and all $f\in
    H^\beta(\cM,\mathcal{N})$ with $\|f\|_{H^\sigma(\cM,\R^{\ell})}\le 1$
    the manifold interpolation $I_{X,\cN}f$ of $f$ is well-defined and satisfies
     \[
     \|f-I_{X,\cN}f\|_{W_q^{\mu}(\cM,\R^{\ell})}\leq
     Ch_{X,\cM}^{\sigma-\mu-m(1/2-1/q)_+}\max\left\{1,\|f\|^{\lfloor
       \sigma\rfloor+1}_{H^\sigma(\cM,\R^{\ell})}+\|f\|_{H^{\sigma}(\cM,\R^{\ell})}\right\}. 
     \]
      for all $0\le \mu\le \ell(\sigma,m,q)$ with $\mu\in\N_0$ if      $q=\infty$.
  \item If $\beta<\sigma$ then there are constants $h_0,C>0$ such that for all
    quasi-uniform $X\subseteq\cM$ with $h_{X,\cM}\le h_0$ and all $f\in
    H^\beta(\cM,\mathcal{N})$ with $\|f\|_{H^\beta(\cM,\R^{\ell})}\le 1$
    the manifold interpolation $I_{X,\cN}f$ of $f$ is well defined and
    satisfies
           \[
    \|f-I_{X,\cN}f\|_{W_q^{\mu}(\cM,\R^{\ell})}\leq
    Ch_{X,\cM}^{\beta-\mu-m(1/2-1/q)_+}\max\left\{1,\|f\|^{\lfloor
      \beta\rfloor+1}_{H^{\beta}(\cM,\R^{\ell})}+\|f\|_{H^{\beta}(\cM,\R^{\ell})}\right\}.
    \]
    for all $0\le \mu\le \ell(\beta,m,q)$ with $\mu\in\N_0$ if $q=\infty$.
  \end{enumerate}

\end{theorem}
\begin{proof}
Let $U=E(V)$ with $V=\{(p,n)\in N\cN: \|n\|_2<\hat{\delta}\}$ for $\hat{\delta}>0$.  
We start with the interpolation operator in the case of matching
smoothness $\beta\ge \sigma$.  According to the error estimate for
vector-valued interpolation from Corollary \ref{Error_M_Vektorwertig}
with $q=\infty$ and $\mu=0$, there exists a constant $\tilde{h}_0>0$
such that for all $X$ with $h_{X,\cM}\leq \tilde{h}_0$ we have 
\[
\|f-I_{X,\R^{\ell}} f\|_{L_{\infty}(\cM, \R^{\ell})}\leq
Ch_{X,\cM}^{\sigma-m/2}\|f\|_{H^{\sigma}(\cM,\R^{\ell})} \le Ch_{X,\cM}^{\sigma-m/2},
  \]
  using also the assumption on $f$. Thus, there is an $0<h_0\le
  \tilde{h}_0$ such that $\|f(x)-I_{X,\R^{\ell}}f(x)\|_2<\hat{\delta}$ by inequality
  \eqref{unendlich_Ungleichung}, showing $I_{X,\R^{\ell}}f(x)\in U$
  for all $x\in \cM$ and hence that the interpolant via projection
  $I_{X,\cN}f=\cP_{\cN}\circ I_{X,\R^{\ell}}f:\cM\to \cN$ is
  well-defined. To derive the stated error estimate, we can argue
similarly as in \cite[Section 2]{ProjectionFE}: As the vector-valued interpolant
only uses point information and the functions $f$ and $I_{X,\cN}f$ coincide on the
sampling points $X$, we have
$I_{X,\R^{\ell}}(I_{X,\cN}f)=I_{X,\R^{\ell}}f$. Hence, applying the
first statement of Corollary 
\ref{Error_M_Vektorwertig} once for $f\in H^{\sigma}(\cM,\R^{\ell})$ and a
second time for $I_{X,\cN}f\in H^{\sigma}(\cM,\R^{\ell})$, we obtain 
\begin{align}
    \|f- I_{X,\cN}f\|_{W_q^{\mu}(\cM,\R^{\ell})}&\leq  
    \|f-I_{X,\R^{\ell}}f\|_{W_q^{\mu}(\cM,\R^{\ell})} +
    \|I_{X,\R^{\ell}}(I_{X,\cN}f)-I_{X,\cN}f\|_{W_q^{\mu}(\cM,\R^{\ell})}\nonumber\\  
    &\leq
    Ch_{X,\cM}^{\sigma-\mu-m(1/2-1/q)_+}\left(\|f\|_{H^{\sigma}(\cM,\R^{\ell})}
    +\|I_{X,\cN}f\|_{H^{\sigma}(\cM,\R^{\ell})} 
    \right)
    \label{ProjectionPaperArgument}.     
\end{align}
By Theorem \ref{thm:ProjectionNormEstimate} we find
\begin{eqnarray*}
\|I_{X,\cN}f\|_{H^{\sigma}(\cM,\R^{\ell})} &=& \|\cP_{\cN}\circ
I_{X,\R^{\ell}}f\|_{H^{\sigma}(\cM,\R^{\ell})}\\
&\le& C
\max\left\{1,\|I_{X,\R^{\ell}}f\|_{H^\sigma(\cM,\R^{\ell})}
\left(1+\|I_{X,\R^{\ell}}f\|_{H^\sigma(\cM,\R^{\ell})}^{\lfloor\sigma\rfloor}\right)\right\} \\ 
&\le& C
\max\left\{1,\|f\|_{H^\sigma(\cM,\R^{\ell})}
\left(1+\|f\|_{H^\sigma(\cM,\R^{\ell})}^{\lfloor\sigma\rfloor}\right)\right\},
\end{eqnarray*}
using also the stability bound (\ref{stabvectorinterpolant}) of the
interpolant. Altogether, we  arrive at the stated
error bound.

The case of non-matching smoothness $\beta<\sigma$ is treated
similarly. Here, we only consider quasi-uniform data sets $X$, meaning
that we assume that the mesh ratio $\rho_{X,\cM}$ is uniformly
bounded, so that the $L_\infty$-error now becomes
\[
\|f-I_{X,\R^{\ell}} f\|_{L_{\infty}(\cM, \R^{\ell})}\leq
Ch_{X,\cM}^{\beta-m/2}\|f\|_{H^{\beta}(\cM,\R^{\ell})} \le Ch_{X,\cM}^{\beta-m/2},
  \]
and we can once again assure that we have $I_{X,\R^{\ell}}f(x)\in U$
for all $x\in\cM$, provided  $h_{X,\cM}$ is sufficiently small. Then,
the second statement of Corollary \ref{Error_M_Vektorwertig} shows
that (\ref{ProjectionPaperArgument}) now takes the form
\[
\|f- I_{X,\cN}f\|_{W_q^{\mu}(\cM,\R^{\ell})}\le
Ch_{X,\cM}^{\beta-\mu-m(1/2-1/q)_+}\left(\|f\|_{H^{\beta}(\cM,\R^{\ell})}
+\|I_{X,\cN}f\|_{H^{\beta}(\cM,\R^{\ell})} 
\right).
\]
An application of Theorem \ref{thm:ProjectionNormEstimate} this time
leads to
\[
\|I_{X,\cN}f\|_{H^{\beta}(\cM,\R^{\ell})} 
\le C
\max\left\{1,\|I_{X,\R^{\ell}}f\|_{H^\beta(\cM,\R^{\ell})}
\left(1+\|I_{X,\R^{\ell}}f\|_{H^\beta(\cM,\R^{\ell})}^{\lfloor\beta\rfloor}\right)\right\}.
\]
To bound $\|I_{X,\R^{\ell}}f\|_{H^\beta(\cM,\R^{\ell})}$ we cannot use
  (\ref{stabvectorinterpolant}) directly, as we are not in the
  reproducing kernel Hilbert space of the employed kernel. However,
  the second error bound from Corollary \ref{Error_M_Vektorwertig}
  yields for $q=2$,  $\mu=\beta$ and quasi-uniform $X$ the bound
  $
  \|f-I_{X,\R^{\ell}} f\|_{H^\beta(\cM,\R^{\ell})} \le C
    \|f\|_{H^\beta(\cM,\R^{\ell})},
   $
    showing also $\|I_{X,\R^{\ell}}f\|_{H^\beta(\cM,\R^{\ell})}\le C
      \|f\|_{H^\beta(\cM,\R^{\ell})}$. 
\end{proof}

Finally, we discuss the manifold-valued penalized least-squares
operator. As before, we only look at matching smoothness requirements
to keep things simple.

\begin{theorem}\label{Final_estimate_PLS}
  Let $\tau>\max\{d/2, (d-m)/2+1\}$. Let $\sigma=\tau-(d-m)/2$ and
  $\beta\ge \max\{m/2,1\}$. Let $U\subseteq\R^{\ell}$ be a uniform
  tubular neighborhood of $\cN$.
  Let $\Phi\in C(\R^d)\cap L_1(\R^d)$
  satisfy (\ref{ftdecay}) and let $\Psi:\cM\times\cM\to\R$ be the
  restricted kernel. 
 Then, there exist constants $C, h_0,\lambda_0>0$ such that for all  quasi-uniform
   $X\subseteq\cM$ with $h_{X,\cM}< h_0$, all $\lambda\le\lambda_0$ and all $f\in
   H^\sigma(\cM,\cN)$ with $\|f\|_{H^\sigma(\cM,\R^{\ell})}\le 1$ the
   penalized least-squares approximation $Q_{X,\lambda,\cN}f$ of $f$
   is well-defined and satisfies
\begin{eqnarray*}
\lefteqn{\|f-Q_{X,\lambda,\cN}f\|_{W_q^\mu(\cM,\R^{\ell})}}\\ & \le &
  Ch_{X,\cM}^{\sigma-\mu-m(1/2-1/q)_+} \left( h_{X,\cM}^{-m/2} +
    \frac{1}{\sqrt{\lambda}}
    h_{X,\cM}^{\sigma-m}\right)\max\left\{1,\|f\|^{\lfloor
       \sigma\rfloor+1}_{H^\sigma(\cM,\R^{\ell})}+\|f\|_{H^{\sigma}(\cM,\R^{\ell})}\right\}\\
    & & \mbox{} +
    Ch_{X,\cM}^{m/\gamma-\mu}\left(\sqrt{\lambda}h_{X,\cM}^{-m/2}+h_{X,\cM}^{\sigma-m}\right)
    \max\left\{1,\|f\|^{\lfloor
       \sigma\rfloor+1}_{H^\sigma(\cM,\R^{\ell})}+\|f\|_{H^{\sigma}(\cM,\R^{\ell})}\right\}. 
\end{eqnarray*}
      for all $0\le \mu\le \ell(\sigma,m,q)$ with $\mu\in\N_0$ if
      $q=\infty$.
\end{theorem}
\begin{proof}
  To simplify the notation in this proof, we will use the abbreviation
  $Q=Q_{X,\lambda,\R^{\ell}}$ and $Q_{\cN}=Q_{X,\lambda,\cN}$. The proof
starts as in the case of the interpolation operator. We 
  first note that Corollary \ref{Error-quasi-M-vector} gives for exact
  data $Y=f(X)$ and quasi-uniform data sets $X$ the 
  $L_\infty$-bound
\[
\|f-Q f\|_{L_\infty(\cM,\R^{\ell})} \le C
\left(h_{X,\cM}^{\sigma-m/2}+ \sqrt{\lambda}\right)
\|f\|_{H^\sigma(\cM,\R^{\ell})} \le C\left(h_{X,\cM}^{\sigma-m/2} +\sqrt{\lambda}\right),
\]
taking also $\|f\|_{H^\sigma(\cM,\R^{\ell})}\le 1$ into account. Thus, 
we can choose an $h_0>0$ and a $\lambda_0>0$ such that
$\|f-Qf\|_{L_\infty(\cM,\R^{\ell})} \le
\hat{\delta}$ for all $\lambda\le\lambda_0$ and $h_{X,\cM}\le
h_0$. Hence, in this situation, the approximation  $Q_{\cN}
f =\cP_{\cN}\circ Qf$ is 
well-defined. Unfortunately, as we do not interpolate the data, we
cannot proceed directly as in the interpolation case. Instead we split
the error into three terms,
\begin{equation}\label{error-split}
f-Q_{\cN} f = (f - Qf) + (Q f -
Q Q_{\cN}f ) + (Q Q_{\cN}
f - Q_{\cN}f),
\end{equation}
and bound each term separately. Bounds on the first and third term are
easily derived from Corollary \ref{Error-quasi-M-vector} using again
exact data $Y=f(X)$ and $Y=Q_{\cN}f(X)$, respectively. For the
first term we immediately have 
\[
\|f-Q f\|_{W_q^\mu(\cM,\R^{\ell})} \le C
\left(h_{X,\cM}^{\sigma-\mu-m(1/2-1/q)_+} + \sqrt{\lambda}h_{X,\cM}^{m/\gamma-\mu}\right)
\|f\|_{H^\sigma(\cM,\R^{\ell})}
\]
and for the third term in (\ref{error-split}) we find 
\[
\|Q_{\cN}f-Q Q_{\cN} f\|_{W_q^\mu(\cM,\R^{\ell})} \le C
\left(h_{X,\cM}^{\sigma-\mu-m(1/2-1/q)_+} + \sqrt{\lambda}h_{X,\cM}^{m/\gamma-\mu}\right)
\|Q_{\cN} f\|_{H^\sigma(\cM,\R^{\ell})}.
\]
Using (\ref{qxl1}) with $Y=f(X)$ shows $\|Q
f\|_{H^\sigma(\cM,\R^{\ell})} \le \|f\|_{H^\sigma(\cM,\R^{\ell})}$ such
that Theorem \ref{thm:ProjectionNormEstimate} gives us once again
\begin{eqnarray*} 
\|Q_{\cN} f\|_{H^\sigma(\cM,\R^{\ell})} & = & \|\cP_{\cN}\circ
Q_{\cN}f\|_{H^\sigma(\cM,\R^{\ell})} \\
& \le &
 C
 \max\left\{1,\|f\|_{H^\sigma(\cM,\R^{\ell})}
 \left(1+\|f\|_{H^\sigma(\cM,\R^{\ell})}^{\lfloor\sigma\rfloor}\right)\right\},
\end{eqnarray*}
which finishes the bound on the third term in (\ref{error-split}). For
the second term, we notice $QQ_{\cN}f = Q \cP_{\cN}Q f$. Hence, if we set
$g=Qf$ and $\widetilde{g}=\cP_{\cN}Qf$ then the second term becomes
$g-Q\widetilde{g}$. Thus, we can apply Corollary
\ref{Error-quasi-M-vector} with $f$ given by $g$ and the data given by
$Y=\widetilde{g}(X)$. This gives the bound
\begin{eqnarray*}
  \|Qf-QQ_{\cN}f\|_{W_q^\mu(\cM,\R^{\ell})} & \le &
  C\left( h_{X,\cM}^{\sigma-\mu-m(1/2-1/q)_+} + \sqrt{\lambda}
    h_{X,\cM}^{m/\gamma-\mu}\right) \|Qf\|_{H^\sigma(\cM,\R^{\ell})} \\
    & & \mbox{} + C\left(h_{X,\cM}^{\sigma-\mu-m(1/2-1/q)_+}
      \frac{1}{\sqrt{\lambda}} + h_{X,\cM}^{m/\gamma-\mu}\right)
      \|Qf-\cP_{\cN}Qf\|_{\ell_2(X)}.
\end{eqnarray*}
As before, we have $\|Qf\|_{H^\sigma(\cM,\R^{\ell})} \le
\|f\|_{H^\sigma(\cM,\R^{\ell})}$. Moreover, as $X$ is quasi-uniform,
the number of points $N$ in $X$ can be bounded by a constant times
$h_{X,\cM}^{-m}$. 
\begin{comment}Next, the closest point projection obviously
satisfies $\|u-\cP_\cN u\|_2 \le \hat{\delta}$ for all $u\in U$ and according to
the considerations at the beginning of this proof, we can choose
$h_{X,\cM}$  and $\lambda$ such that 
$\hat{\delta} \ge C (h_{X,\cM}^{\sigma-m/2} +
\sqrt{\lambda})\|f\|_{H^\sigma(\cM,\R^{\ell})}$. This all together
yields \textcolor{red}{Meines Erachtens steht $\mathcal{P}_{\cM}Qf$ auf der rechten Seite der folgenden Ungleichung UND die Ungleichungen gehen doch in die komplett falschen Richtungen...}
\end{comment}
As $\|Qf(x)-\mathcal{P}_{\cM}Qf(x)\|_2\leq \|Qf(x)-f(x)\|_2$ for all $x\in \Omega$ due to the definition of the closest point projection, we get
\begin{eqnarray*}
  \|Qf-\cP_{\cN} Qf\|_{\ell_2(X)} & \le & C h_{X,\cM}^{-m/2}
  \|Qf-f\|_{\ell_\infty(X)}\\
  & \le& C h_{X,\cM}^{-m/2}\left(h_{X,\cM}^{\sigma-m/2} +
  \sqrt{\lambda}\right)\|f\|_{H^\sigma(\cM,\R^{\ell})},
\end{eqnarray*}
where we also used the arguments at the beginning of this proof.
This leads to the bound
\begin{eqnarray*}
  \|Qf-QQ_{\cN}f\|_{W_q^\mu(\cM,\R^{\ell})} & \le &
  Ch_{X,\cM}^{\sigma-\mu-m(1/2-1/q)_+} \left( h_{X,\cM}^{-m/2} +
    \frac{1}{\sqrt{\lambda}}
    h_{X,\cM}^{\sigma-m}\right)\|f\|_{H^\sigma(\cM,\R^{\ell})}\\
    & & \mbox{} +
    Ch_{X,\cM}^{m/\gamma-\mu}\left(\sqrt{\lambda}h_{X,\cM}^{-m/2}+h_{X,\cM}^{\sigma-m}\right)
    \|f\|_{H^\sigma(\cM,\R^{\ell})}
\end{eqnarray*}
when collecting only the dominant terms. All this together gives the
stated error bound.
\end{proof}

To better understand this bound, let us, for simplicity, look at the
case $q=2$. Using the additional assumption that
$\|f\|_{H^\sigma(\cM,\R^{\ell})} \le 1$, we have with $\gamma=2$ and $h=h_{X,\cM}$,
\begin{eqnarray*}
\|f-Q_{X,\lambda,\cN}f\|_{H^\mu(\cM,\R^{\ell})} &\le& C
h^{\sigma-\mu}\left(h^{-m/2} + \frac{1}{\sqrt{\lambda}}
h^{\sigma-m}\right) + C h^{m/2-\mu}\left(\sqrt{\lambda} h^{-m/2}
+h^{\sigma-m}\right)\\
& \le & C\left( h^{\sigma-\mu-m/2} + \frac{1}{\sqrt{\lambda}}
  h^{2\sigma-\mu-m} + \sqrt{\lambda} h^{-\mu}\right).
\end{eqnarray*}
Thus, the choice $\lambda = h^{2\sigma-m}$ leads finally the bound
\[
\|f-Q_{X,\lambda,\cN}f\|_{H^\mu(\cM,\R^{\ell})} \le C h^{\sigma-\mu-m/2},
\]
which gives convergence for $0\le \mu\le \sigma-m/2$, though it holds
for $0\le \mu\le \sigma$. The unexpected loss of $m/2$ in the
approximation order is a result of bounding the $\ell_2(X)$-norm by
$h^{-m/2}$ times the $\ell_\infty(X)$-norm.

Let us conclude our theoretical discussion with a few remarks
concerning these error estimates.
In general, it is not possible to quantify the radius $\hat{\delta}$
of the uniform tubular neighborhood $U$ of $\cN$.  
However, as we shall see in the next section, such estimates for
$\hat{\delta}$ can sometimes be obtained once a specific manifold
$\cN$ is fixed.  
Moreover, even when $\hat{\delta}$ is known, we cannot determine how
small the constant $h_0$ must be chosen, since the norm
$\|f\|_{H^{\sigma}(\cM,\R^{\ell})}$ of the unknown function $f$ is typically
not available.  

Depending on the particular choice of the compact embedded submanifold
$\cN$, one must also address how the projection operator $\cP_\cN$ can
be implemented numerically.  
For certain manifolds, efficient methods to compute best
approximations for points within tubular neighborhoods exist.  
An example will be presented in Section \ref{Numerics}.

The choices $q=\infty$ and $\mu=0$ in our error estimates, give bounds
for the pointwise error. Because 
of Definition \ref{Vector_valued_Sobolev} this error is measured in
the Euclidean norm of the ambient space $\R^{\ell}$ of $\cN$. However, for a
connected, compact embedded submanifold $\cN\subseteq\R^{\ell}$ this
distance is strongly equivalent to the intrinsic distance
$d_{\cN}:\cN\times \cN\to \R$ according to \cite[Theorem
  6]{Fuselier-Wright-12-1}, so that we can expect the same convergence order in this metric. More generally, for compact manifolds arbitrary Riemannian metrics induce strongly equivalent distances and our error bounds hold analogously.

So far we have only considered embedded submanifolds of Euclidean
spaces. However, our method can easily be generalized to arbitrary
compact manifolds $\cM$ and $\cN$. The Nash embedding theorem
\cite{Nash} guarantees that every $m$-dimensional compact Riemannian
manifold can be isometrically embedded into some $\R^d$ with $d \leq
m(3m+11)/2$. The image $\iota(\cM)\subseteq \R^d$ of such an embedding
is an $m$-dimensional submanifold of $\R^d$, which is compact if $\cM$
is compact. Choosing isometric embeddings $\iota_1:\cM\to \R^{d}$ and
$\iota_2:\cN\to \R^{\ell}$ 
we can again restrict positive definite kernels on the ambient space
$\R^{d}$ to positive definite kernels on $\iota_1(\cM)$. As $\iota_1$
is an isometry, we get $h_{X,\cM}=h_{\iota_1(X),\iota_1(\cM)}$. At the
same time, we can choose a uniform tubular neighborhood $U$ of
$\iota_2(\cN)\subseteq\R^{\ell}$, as well as an associated closest
point projection. One obtains results analogous to those in
Theorems~\ref{Final_estimate} and \ref{Final_estimate_PLS}.
However, note that the choice of an embedding for a given compact manifold is far from trivial in general and obviously highly depends on the specific manifold.
Even if such an embedding is known explicitly, it is not necessarily useful in practice: It is important that the corresponding tubular neighborhood is as big as possible, while at the same time the associated closest point projection is realizable numerically. Fortunately, for some manifolds appropriate embeddings and corresponding numerical realizations of the closest point projection are known. We quickly mention such an example in the following section. Of course, one also has to keep in mind the curse of dimensionality, when embedding the manifold $\cM$ via the embedding $\iota_1$.

%%%%%%%%%%%%%%%%%%%%%%%%%%%%%%%%%%%%%%%%%%%%%%%%%

\section{Numerical examples}\label{Numerics}
We conclude this article with several numerical examples that
illustrate our method and the theoretical error bounds established in
Theorem~\ref{Final_estimate}. In the first of the two following
sections, we consider two academic examples, while in the second
section we apply our methodology to a real-world example from
crystallography taken from \cite{HielscherLippert}. 

\subsection{Academic examples involving Stiefel manifolds}
We consider the Stiefel manifold 
$\St(n,r):=\{S\in \R^{n\times r}: S^\transpose S=I_r\}$
of column-orthogonal real matrices in $\R^{n\times r}$,  
where obviously $n\geq r$. (We can identify $\R^{n\times r}$ with
$\R^{nr}$ to perfectly fit the situation from before). For $n=r$, this
set reduces to the orthonormal matrices $\St(n,n)=O(n)$. For $r=1$, we
obtain the unit sphere $\St(n,1)=S^{n-1}\subseteq\R^n$. As
$\St(n,r)=\Phi^{-1}(\{0_{r\times r}\})$ with 
\[
\Phi:\R^{n\times r}\to \text{sym}(r),\quad S\mapsto S^TS-I_r,
\]
and $\Phi$ is a smooth submersion, $\St(n,r)$ is an embedded
submanifold of $\R^{n\times r}$ of dimension $nr-r(r+1)/2$ according
to the Submersion level set theorem \cite[Corollary
  5.13]{LeeSmooth}. We also conclude that $\St(n,r)$ is closed and
hence compact by Heine-Borel, as $\|S\|_F=\sqrt{\text{tr}(S^\transpose
  S)}=1$
for $S\in \St(n,r)$. Therefore, our method is applicable for domain
$\cM$ and codomain $\cN$ chosen as Stiefel manifolds. 

For a first example we consider a function between
$\cM=S^1\subseteq\R^2$ and $\cN=S^2\subseteq\R^3$. We have already
observed that in this case, there is a uniform tubular neighborhood
$U$ of size $\hat{\delta}=1$ and the associated closest point
projection reduces to normalization, $\cP_{\cN}(v)=v/\|v\|_2$ for
$v\in U$. For $\varepsilon>0$ we consider the function 
\[
f_{\varepsilon}:S^1\to S^2,\quad
f_{\varepsilon}(x,y)=S(\|(x,y)-(1,0)\|_2^{\varepsilon},
3\|(x,y)-(1,0)\|_2^{\varepsilon}),
\]
where $S(\theta,\phi)$ denotes the point on the unit sphere $S^2$ with
spherical coordinates $(\theta,\varphi)$, using the convention that
$\theta$ is the polar angle and $\varphi$ the azimuthal angle. We
extend the spherical coordinate map to all of $\R\times
\R$ in the obvious periodic way. Extended to $\R^2$ the
function $f_{\varepsilon}$ is contained in $H^{2.5}(\R^2,\R^2)$ for
$\varepsilon>1.5$, and thus  $f_{\varepsilon}\in H^{2}(S^1, S^2)$ for
$\varepsilon>1.5$. For $\varepsilon=1.51$, we sample $f_{\varepsilon}$
at $N\in\{5\cdot 2^{i}: 0\leq i\leq 11\}$ points on $S^1$, which are
obtained by mapping equidistant points on the interval $[0,2\pi]$ to
$S^1$ via polar coordinates. We interpolate $f_{\varepsilon}$ based on
this data and a restricted positive definite kernel
$\Psi=K|_{\cM\times\cM}$ built from the Wendland function
$\phi_{2,1}$, i.e.,  
\[
K(x,y)=\Phi_{2,1}(x-y)=\phi_{2,1}(\|x-y\|_2)=\max\{0,
(1-\|x-y\|_2)^4\}\cdot (4\|x-y\|_2+1).
\]
With the notation from earlier, we have $\tau=d/2+k+1/2=2.5$,
$\sigma=\tau-(d-m)/2=2$, and hence expect a convergence of order $2-\mu-
(1/2-1/q)_+$ for the error measured in the
$W_q^{\mu}(S^1,\R^3)$-norm. For the $L_{\infty}$-error this means a
convergence of order $\cO(h_{X,\cM}^{1.5})$. The approximation
errors for our interpolant $I_{X,\cN}f_{\varepsilon}$ and the
interpolant $I_{X,\R^{\ell}}f_{\varepsilon}$ in the ambient space are
displayed in Figure \ref{fig:ex1errors}. Convergence orders are
numerically estimated in Table \ref{tab:filldistance} and match our
theoretical results. Here, the errors were estimated with the help of
$30000$ validation points, created analogously to the sample points
used for interpolation. Also note that both interpolants do not differ
significantly anymore, once the fill distance is below a certain
threshold. Nevertheless, the ambient interpolant
$I_{X,\R^{\ell}}f_{\varepsilon}$ will typically not take values in $S^2$
outside the sample locations.  

\begin{figure}[h!]
\centering

\begin{subfigure}{0.48\textwidth}
\centering
\includegraphics[width=\linewidth]{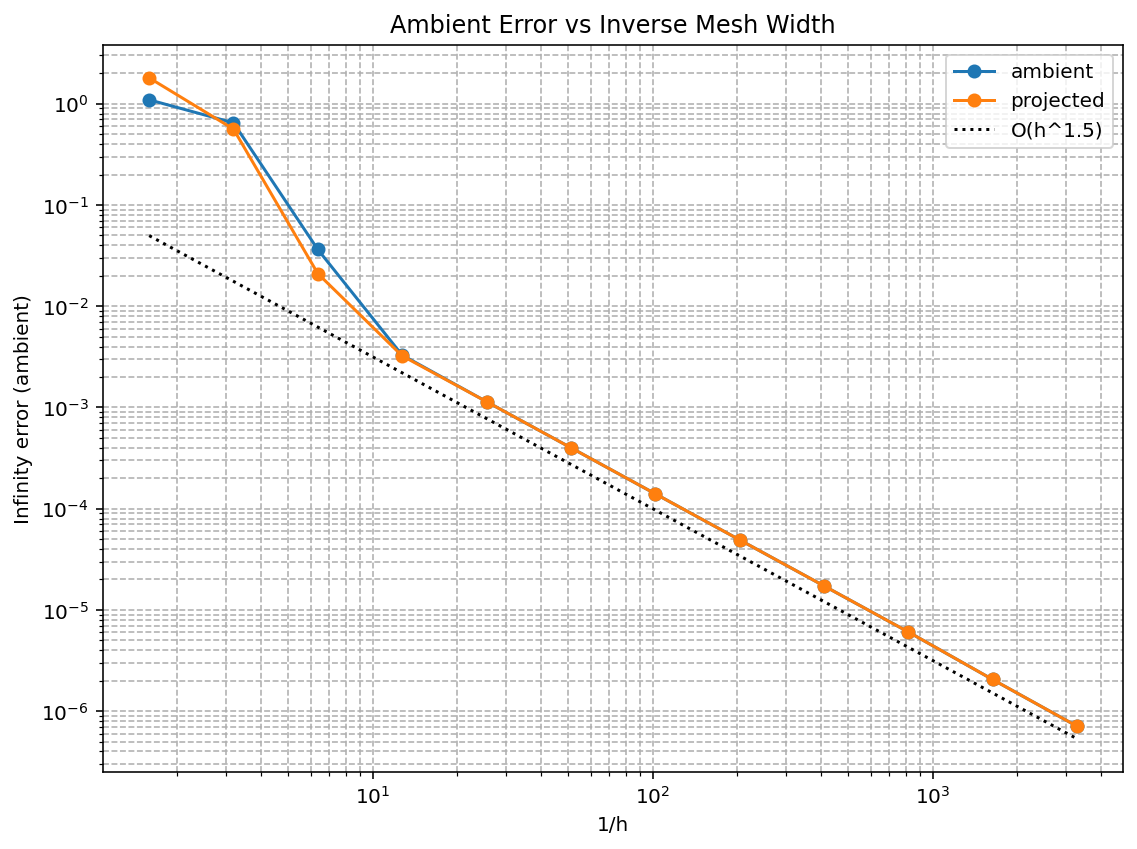}
\vfill
\caption{$L_{\infty}(S^1,\R^3)$ error of the ambient
  interpolant $I_{X,\R^{\ell}}f_{\varepsilon}$ and the projected
  interpolant $I_{X,\cN}f_{\varepsilon}$ plotted against the estimated
  fill distance.} 
\label{fig:ex1errors}
\end{subfigure}
\hfill
\begin{subfigure}{0.48\textwidth}
\centering
\begin{tabular}{c c c}
\hline
$N$ & $h_{X,\Omega}$ & order \\
\hline
5     & $6.28318531\times10^{-1}$ & \\
10    & $3.14159265\times10^{-1}$ & 1.68180654 \\
20    & $1.57079633\times10^{-1}$ & 4.74012179 \\
40    & $7.85398163\times10^{-2}$ & 2.69909133 \\
80    & $3.91651884\times10^{-2}$ & 1.50246076 \\
160   & $1.95825942\times10^{-2}$ & 1.51022242 \\
320   & $9.79129710\times10^{-3}$ & 1.51054018 \\
640   & $4.89564855\times10^{-3}$ & 1.51005035 \\
1280  & $2.44782428\times10^{-3}$ & 1.51000746 \\
2560  & $1.22391214\times10^{-3}$ & 1.51000014 \\
5120  & $6.11956069\times10^{-4}$ & 1.56084848 \\
10240 & $3.05978035\times10^{-4}$ & 1.53119225 \\
\hline
\end{tabular}
\vfill
\caption{Numbers of samples, estimated fill distances and estimated convergence orders for interpolating $f_{\varepsilon}$ by $I_{X,\cN}f_{\varepsilon}$.}
\label{tab:filldistance}
\end{subfigure}

\end{figure}

Interestingly, uniform tubular neighborhoods and closest point
projections are also known for $\St(n,r)$ with general parameters
$n\geq r$. According to Theorem 4.1 from \cite{Higham}, if $A\in
\R^{n\times r}$ is a full rank matrix, there exists a unique closest
point (with respect to the norm of the ambient $\R^{nr}$, i.e., the
Frobenius norm) to $A$ from $\St(n,r)$, which is simply the factor
$S\in \St(n,r)$ in the polar decomposition $A=SP$ (Recall that each
full rank matrix $A\in \R^{n\times r}$, $n\geq r$, has a unique
decomposition $A=SP$ into a matrix $S\in \St(n,r)$ and a symmetric
positive definite matrix $P\in \R^{r\times r}$). The size of a uniform
tubular neighborhood can also be quantified. According to
\cite[Theorem 7.41]{Wendland-17-2}, the distance of a matrix $S$ with non-zero
singular values $\sigma_1\geq\sigma_2\geq\dots,\geq \sigma_r$ to the
closest matrix of rank $k<r$ measured in the Frobenius norm is given
by $\left(\sum_{j=k+1}^r \sigma_j^2\right)^{1/2}$.  As a Stiefel matrix $S\in \St(n,r)$
has all singular values equal to $1$, it follows that the closest
 matrix $W\in \R^{n\times r}$ to $S$
 which does not have full rank, fulfills $\|S-W\|_F=\sigma_r =1.$
 %As the
 %Frobenius norm is always bigger than or equal to the $2$-norm,
 Thus, there
exists a uniform tubular neighborhood of $\St(n,r)$ of size
$\hat{\delta}=1$. Overall, for all matrices $A\in \R^{n\times r}$
which have a Frobenius distance less than $1$ to $\St(n,r)$, we can
compute a unique closest point from $\St(n,r)$ with the help of the
polar decomposition. 

This is used in our second  example, where we interpolate a function
$g$ defined on $S^1\subseteq\R^2$, which takes values in
$\St(4,2)$. To construct an interesting example, we use that the
tangent space at $A\in \St(n,r)$ can be characterized as 
$$T_A\St(n,r)=\{\Delta = AV+A_{\perp}K: V\in \text{skew}(r), K\in \R^{(n-r)\times r}\}\subseteq\R^{n\times r},$$
where $A_{\perp}\in \R^{n\times(n-r)}$ is an orthogonal matrix whose columns span the orthogonal complement of the range of $A$. The corresponding Riemannian exponential $\exp_A:T_A\St(n,r)\to \St(n,r)$ is given by 
$$\exp_A(\Delta) = (A,A_{\perp})\exp_m\left(\begin{pmatrix}2V& -K^T\\
    K & 0\end{pmatrix}\right)I_{n\times r}\exp_m(-V),$$
    where $I_{n\times r}:=\begin{pmatrix} I_r\\
    0_{(n-r)\times r}
    \end{pmatrix}\in \R^{n\times r}$ and $\exp_m:\R^{n\times n}\to \R^{n\times n}$ denotes the matrix exponential.
Therefore, for the function
 $$g:S^1\to \St(4,2),\quad x\mapsto \exp_A\left(Af_1(x)\cdot \begin{pmatrix}
     0 & 1\\
     -1 & 0
 \end{pmatrix}+A_{\perp}f_2(x)\cdot \begin{pmatrix}
     1 & 0\\
     0 & 1
 \end{pmatrix}\right)$$
 with $f_1=f_2:S^1\to \R,\ x\mapsto \|(x,y)-(1,0)\|_2^{\varepsilon}$
 and $A=I_{n\times r}$, we have $g\in H^{4}(S^1, \St(4,2))$ for
 $\varepsilon>3.5$. We repeat our procedure with $\varepsilon = 3.51$,
 sample sets of size $N\in\{5\cdot 2^{i}: 0\leq i\leq 8\}$ and a
 validation set with $20000$ points, both created as before. This time
 we use the Wendland kernel $\Phi_{2,3}$, whose restricted native
 space corresponds to functions in $H^4(S^1,\St(4,2))$. According to
 Theorem \ref{Final_estimate} we expect convergence order
 $\cO(h_{X,\Omega}^{3.5})$ for the pointwise error. Corresponding
 results are displayed in Figure \ref{fig:ex2errors} and Table
 \ref{tab:filldistance2}. Again, the results match our theoretical
 expectations.

 \begin{figure}[h!]
\centering
\begin{subfigure}{0.48\textwidth}
\centering
\includegraphics[width=\linewidth]{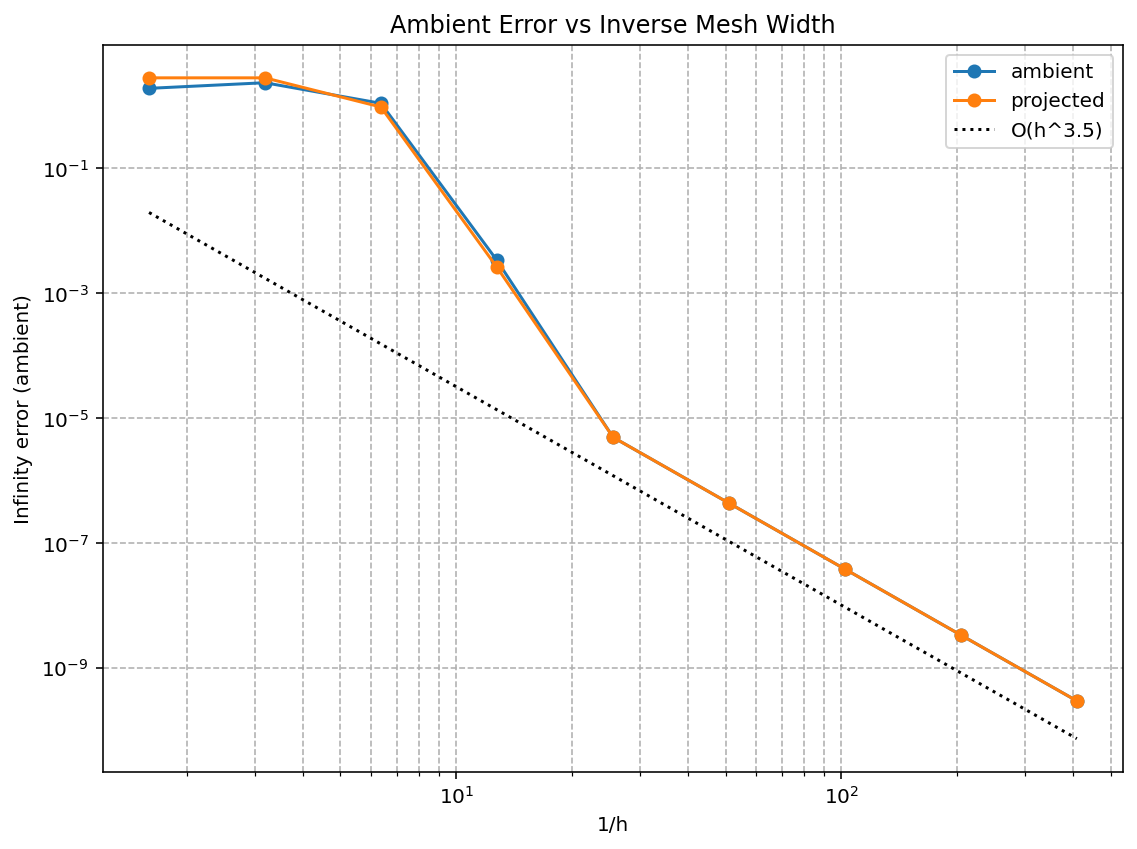}
\caption{$L_{\infty}(S^1,\St(4,2))$ error of the ambient interpolant $I_{X,\R^{\ell}}g$ and the projected interpolant $I_{X,\cN}g$ plotted against the estimated fill distance.}
\label{fig:ex2errors}
\end{subfigure}
\hfill
\begin{subfigure}{0.48\textwidth}
\centering
\begin{tabular}{c c c}
\hline
$N$ & $h_{X,\Omega}$ & order \\
\hline
5     & $0.62831853$ & \\
    10    & $0.31415927$ & $0.00000007$\\
    20    & $0.15707963$ & $1.55245434$  \\
    40    & $0.07853982$ & $8.51857713$ \\
    80    & $0.03916519$ & $9.03296693$ \\
    160   & $0.01958259$ & $3.50802392$ \\
    320   & $0.00979131$  & $3.50982539$ \\
    640   & $0.00489565$ & $3.50989715$\\
    1280  & $0.00244782$ & $3.50380197$ \\
\hline
\end{tabular}
\caption{Numbers of samples, estimated fill distances and estimated convergence orders for interpolating $g$ by $I_{X,\cN}g$.}
\label{tab:filldistance2}
\end{subfigure}

\end{figure}

\subsection{Wave velocities in crystalline materials}
We conclude our numerical section with an example from
crystallography, which originally stems from \cite[Section
  4.1]{HielscherLippert}. In seismology, geologic properties of the
earth can be investigated by analyzing the propagation of earthquake
waves. Their propagation velocity and polarization direction depend on
the propagation direction, i.e., the direction in which the wave
travels, relative to the crystal lattice of the minerals in the mantle
of the earth. This directional dependence is commonly described by the
Christoffel equation from linear elasticity, an eigenvalue problem
involving the Christoffel tensor, which itself is derived from the
effective stiffness tensor of the polycrystalline aggregate. The
effective stiffness tensor is obtained by rotating and averaging the
single-crystal stiffness tensors (in our case, those of the silicate
mineral olivine) according to a prescribed orientation distribution
function. 

Solving this eigenvalue problem for propagation directions in $S^2$
(for example conveniently with the help of the open-source MATLAB
toolbox MTEX), one ends up with a function $f : S^2 \to \R P^2,$ 
which maps each propagation direction to the corresponding
polarization direction. Here, $\R P^2$ denotes the real
projective plane, that is, the space of all lines through the origin
in $\R^3$. In \cite[Section 4.1]{HielscherLippert}, this
function is reconstructed for given material properties from a finite
number of measurements. To this end, $\R P^2$ is identified
with $S^2/_ {\sim}$ via the equivalence relation $x\sim -x$, and
afterwards embedded into the space of real symmetric matrices
$\text{sym}(3)\subseteq\R^{3\times 3}$ via the embedding $\iota:
S^2/_{\sim}\to \text{sym}(3),\ \iota(x)=xx^T.$ It is shown in
\cite[Lemma 4.1]{HielscherLippert} that  
\[
U=E(\{(p,n)\in N(\iota(S^2/_{\sim})): \|n\|_F<\hat{\delta}\})
\]
with $\hat{\delta}=\frac{1}{\sqrt{2}}$ is a uniform tubular
neighborhood of $\iota(S^2/_{\sim})$. Furthermore, the closest point
projection $\cP_{\cN}: U\subseteq\text{sym}(3) \to
\cN=\iota(S^2/_{\sim})$ with respect to the Frobenius norm is given by 
\[
\cP_{\cN}(B)=uu^T,\quad B\in \sym(3),
\]
where $u\in \R^3$ is a normalized eigenvector to the eigenvalue of $B$
which has the largest absolute value. Calculating an ambient
interpolant in $\text{sym(3)}$ via spherical harmonics and using the
closest point projection $\cP_{\cN}$, the authors reconstruct $f$ from
its values at 144 Chebyshev quadrature nodes. We repeat this
calculation with our interpolation method, using a Wendland kernel
$\Phi_{3,1}$ without further scaling. In Figure \ref{fig:Vergleich},
we show the pointwise error (measured in the Frobenius norm) of the
approximation method from \cite{HielscherLippert} and of our
interpolation $I_{X,\cN}f$, and compare the difference between both
approximations for directions in the upper hemisphere of $S^2$. For
computing the approximation via spherical harmonics we used the
implementation made available in \cite{HielscherLippert}. The results
provide cautious indications that, despite its simplicity, our
approximation technique can keep pace with, or even surpass, spherical
harmonic interpolation. However, a more definitive assessment will
require extensive further experiments. 

\begin{figure}[ht]
    \centering
    \begin{subfigure}[t]{0.32\textwidth}
        \centering
        \includegraphics[width=\textwidth]{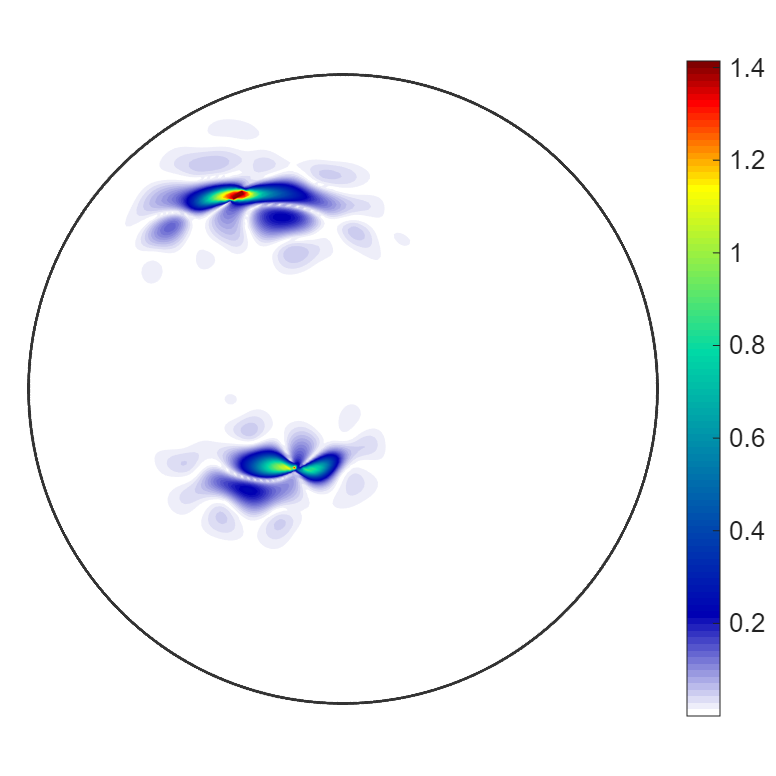}
        \caption{Pointwise error in Frobenius norm of interpolant $I_{X,\cN}f$ for propagation directions in the upper hemisphere.}
    \end{subfigure}
    \hfill
    \begin{subfigure}[t]{0.32\textwidth}
        \centering
        \includegraphics[width=\textwidth]{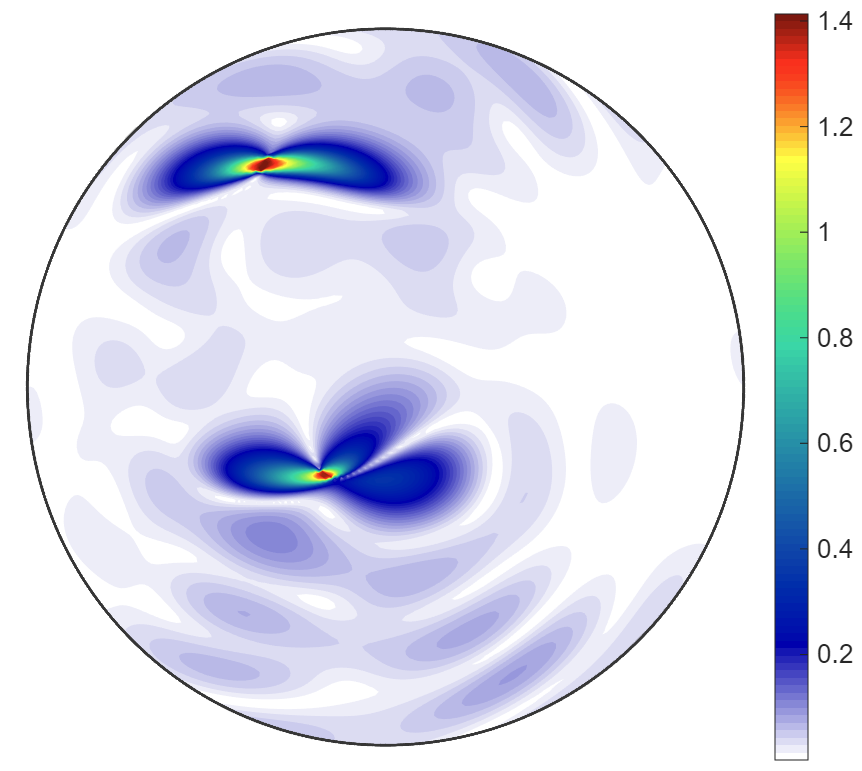}
        \caption{Pointwise error in Frobenius norm of reconstruction of $f$ via spherical harmonics from \cite{HielscherLippert} for propagation directions in the upper hemisphere.}
    \end{subfigure}
    \hfill
    \begin{subfigure}[t]{0.32\textwidth}
        \centering
        \includegraphics[width=\textwidth]{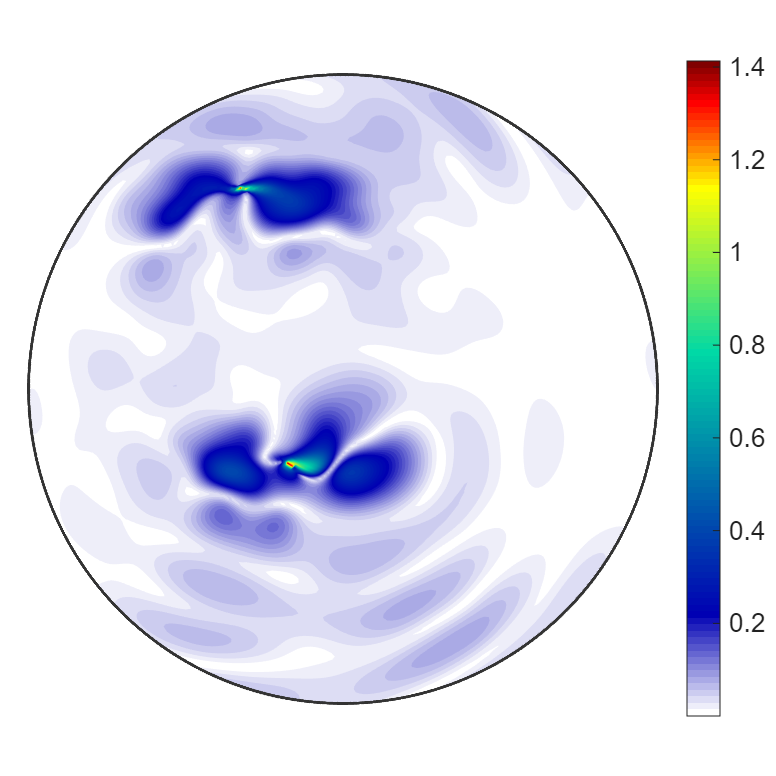}
        \caption{Pointwise Frobenius distance between both approximation methods for propagation directions in the upper hemisphere.}
    \end{subfigure}

    \caption{Reconstructing a function $f:S^2\to \R P^2$ from the analysis of crystalline materials.}
    \label{fig:Vergleich}
\end{figure}

%\appendix

\end{document}